\documentclass[a4paper,11pt,reqno]{amsart}
\usepackage{color}
\usepackage{enumerate}
\usepackage{amssymb, amsmath}

\theoremstyle{plain}
\newtheorem{theorem}{Theorem}[section]
\newtheorem{corollary}[theorem]{Corollary}
\newtheorem{lemma}[theorem]{Lemma}
\newtheorem{proposition}[theorem]{Proposition}
\newtheorem{definition}[theorem]{Definition}

\newtheorem*{definition*}{Definition}

\theoremstyle{remark}
\newtheorem{remark}[theorem]{Remark}
\newtheorem{example}[theorem]{Example}

\newtheorem*{claim*}{Claim}
\newtheorem*{remark*}{Remark}
\newtheorem*{example*}{Example}
\newtheorem*{notation*}{Notation}

\def\N{{\mathbb N}}

\def\R{{\mathbb R}}

\def\PP{{\mathbb P}}
\def\EE{{\mathbb E}}
\newcommand{\K}{\mathrm{k}}
\newcommand{\KK}{\mathrm{K}}
\newcommand{\LL}{\mathrm{L}}
\newcommand{\RR}{\mathrm{R}}
\newcommand{\HH}{\mathrm{H}}

\newcommand{\Dom}{\mathit{Dom}}
\newcommand{\A}{\mathcal{A}}
\newcommand{\E}{\mathcal{E}}
\newcommand{\BE}{\mathrm{BE}}
\newcommand{\CD}{\mathrm{CD}}
\newcommand{\Ric}{\mathrm{Ric}}
\newcommand{\diam}{\mathrm{diam}}
\newcommand{\tr}{\mathrm{tr}}

\numberwithin{equation}{section}

\begin{document}
\title[Ricci Tensor for Diffusion Operators and Curvature-Dimension Inequalities]{Ricci Tensor for Diffusion Operators and Curvature-Dimension Inequalities under Conformal Transformations and Time Changes}
\author{Karl-Theodor Sturm}
 \address{
 University of Bonn\\
 Institute for Applied Mathematics\\
 Endenicher Allee 60\\
 53115 Bonn\\
 Germany}
  \email{sturm@uni-bonn.de}

\thanks{
The author gratefully acknowledges  support by the European Union through the
ERC Advanced Grant ``Metric measure spaces and Ricci curvature - analytic, geometric, and probabilistic challenges" (``RicciBounds'')
as well as support by the German Research Foundation through the Hausdorff Center for Mathematics and the Collaborative Research Center 1060 ``The Mathematics of Emergent Effects''.
}

\maketitle

\begin{abstract}
Within the $\Gamma_2$-calculus of Bakry and Ledoux, we define the Ricci tensor induced by a diffusion operator, we deduce precise formulas for its
behavior under drift transformation, time change and conformal transformation, and we derive new transformation results for the curvature-dimension conditions of Bakry-{\'E}mery as well as for those of Lott-Sturm-Villani.
Our results are based on new identities and sharp estimates for the $N$-Ricci tensor and for the Hessian. In particular, we obtain Bochner's formula
in the general setting.
\end{abstract}

\section{Introduction}
Generators of Markov diffusions -- i.e. Markov processes with continuous sample paths -- are second-order differential operators $\LL$ or suitable generalizations of them.
These generators give rise to intrinsically defined geometries on the underlying spaces $X$. Equilibration and regularization properties of such stochastic processes are intimately linked to curvature bounds for the induced geometries.

We regard $\Gamma(u,v)(x)=\frac12[{\LL}(uv)-u{\LL}v-v{\LL}u](x)$ as the `metric tensor' at $x\in X$ and
\begin{eqnarray}
\RR(u,u)(x)=\inf\big\{ \Gamma_2(\tilde u,\tilde u)(x): \ \Gamma(\tilde u- u)(x)=0\big\}
\end{eqnarray}
as the 'Ricci tensor' at $x$ where
$\Gamma_2(u,v)=\frac12[{\LL}\Gamma(u,v)-\Gamma(u,{\LL}v)-\Gamma(v,{\LL}u)]$.
More generally, for $N\in[1,\infty]$ -- which will play the role of an upper bound for the dimension -- we consider the $N$-Ricci tensor
 \begin{eqnarray*}
\RR_N(u,u)(x)=\inf\big\{ \Gamma_2(\tilde u,\tilde u)(x)-\frac1N \big(\LL \tilde u\big)^2(x): \ \Gamma(\tilde u- u)(x)=0\big\}.
\end{eqnarray*}
If $\LL$ is the Laplcae-Beltrami operator on a complete $n$-dimensional Riemannian manifold then
\begin{eqnarray*}
  \RR_N(u,v)=\Ric(\nabla u,\nabla v)
 \end{eqnarray*}
for all $N\in[n,\infty]$ and all $u,v$.
One of the key results of this paper is \emph{Bochner's formula
} in this general setting.
\begin{theorem}
Under mild regularity assumptions, for each $u$
\begin{eqnarray}\Gamma_2(u,u)=\RR(u,u)+\big\|\HH_u(.)\big\|_{HS}^2\end{eqnarray}
where $\HH_u(.)$ denotes the {Hessian} of $u$ and $\big\|\,.\,\big\|_{HS}$ its {Hilbert-Schmidt norm}.
\end{theorem}
A refined assertion states that
\begin{eqnarray*}
\Gamma_2(u,u)
&=&\RR_N(u,u)+\big\| \HH_u(.) \big\|^2_{HS}+
\frac1{N-n}\Big( \tr\, \HH_u(.) - \, \LL u\Big)^2
\end{eqnarray*}
whenever $N$ is larger than the vector space dimension of the `tangent space' defined in terms of $\Gamma$.
In particular, this equality implies \emph{Bochner's inequality}
$\Gamma_2(u,u)
\ge\RR_N(u,u)+\frac1N \big(\LL u\big)^2$.
The latter in turn implies
 the \emph{energetic curvature-dimension condition} or  \emph{Bakry-{\'E}mery condition}
\begin{equation}\label{BE}
\Gamma_2(u,u)\ge \K\cdot\Gamma(u,u)+\frac1N ({\LL}u)^2
\end{equation}
for every function $\K: X\to\R$ which is  a pointwise lower bound for the $N$-Ricci tensor in the sense that
 $\RR_N\ge \K\cdot\Gamma$. 
 
 \bigskip

The second major topic  of this paper is to study the transformation of the $N$-Ricci tensor and of the
the Bakry-{\'E}mery condition
under each of the basic transformations of stochastic processes.
The transformations which we have in mind are:
\begin{itemize}
\item time change ${\LL}' u=f^2\,{\LL}u$
\item drift transformation ${\LL}' u={\LL}u+ \Gamma(h,u)$
\item metric transformation ${\LL}' u=f^2\,{\LL}u+ \Gamma(f^2,u)$
\item conformal transformation ${\LL}' u=f^2\,{\LL}u- \frac{N-2}2\Gamma(f^2,u)$
\item Doob transformation  ${\LL}' u=\frac1{\rho}{\LL}(u\rho)$ provided ${\LL}\rho=0$ and $\rho>0$.
\end{itemize}
Indeed, we will study more general transformations of the form
\begin{equation}\label{Lsharp}
{\LL}' u=f^2\, {\LL}u+f\,\sum_{i=1}^r g_i \, \Gamma(h_i,u)
\end{equation}
which will cover all the previous  examples and, beyond that, provides various new examples (including non-reversible ones).
Our main result is

\begin{theorem}
For every $N'>N$ the $N'$-Ricci tensor for the operator $\LL'$ from   \eqref{Lsharp}
can be pointwise estimated from below in terms of the $N$-Ricci tensor for $\LL$ and
 the functions  $f,g_i,h_i$.
For instance, in the case of the time change
\begin{eqnarray}
\RR'_{N'}\ge f^2\,\RR_N +
\frac12{\LL}f^2-2\Gamma(f)^2- \frac{(N-2)(N'-2)}{N'-N}\,\Gamma(f,.)^2.
\end{eqnarray}
In the particular case of the conformal transformation, such an estimate is also available for $N'=N$ and in this case
it indeed is an \emph{equality}.
\end{theorem}

\begin{corollary}
Assume that ${\LL}$ satisfies the condition $\BE(k,N)$ from \eqref{BE}. Then for every $N'>N$ the transformed operator ${\LL}'$ from   \eqref{Lsharp}
satisfies the condition $\BE(\K',N')$ with a function $\K'$ explicitly given in terms of the functions $\K$ and $f,g_i,h_i$.
For instance, in the case of the time change
\begin{eqnarray}
\K'= f^2\,\K +
\frac12{\LL}f^2-N^*\,\Gamma(f,f).
\end{eqnarray}
with $N^*=2+\frac{[(N-2)(N'-2)]_+}{N'-N}$.
In the particular case of the conformal transformation, such a Bakry-{\'E}mery condition is also available for $N'=N$.
\end{corollary}

The Bakry-{\'E}mery condition is   particularly useful and mostly applied in cases where the Markov diffusion is reversible w.r.t. some measure $m$ on the state space $X$ and where
 ${\LL}$ is the generator of the Dirichlet form
$\E(u)=\int \Gamma(u,u)\,dm$ on $L^2(X,m)$.
In this framework, most of the previous examples for transformations of generators can be obtained as generators ${\LL}'$ of
Dirichlet forms
\begin{equation}\label{dir-form}
\E'(u)=\int \Gamma(u,u)\,\phi^2\,dm\quad\mbox{on }L^2(X,\rho^2\,m)
\end{equation}
for suitable choices of the weights $\phi$ and $\rho$:
time change ($\phi=1$, $\rho=1/f$), drift transformation/Doob transformation ($\phi=\rho=e^{h/2}$), metric transformation ($\rho=1$, $\phi=f$), conformal transformation ($\rho=f^{-N/2}$, $\phi=f^{-N/2+1}$).

The study of
Bakry-{\'E}mery estimates for such Dirichlet forms is closely related to the analysis of curvature bounds in the sense of Lott-Sturm-Villani for metric measure spaces.
We say that a mms $(X,d,m)$ satisfies the \emph{entropic curvature-dimension condition} $\CD^e(\KK,N)$ for (extended) real parameters $\KK$ and $N$ if the Boltzmann entropy $S$ (with $S(\mu)=\int\rho\log\rho\,dm$ if $\mu=\rho\,m$) satisfies
\begin{equation}
D^2 S-\frac1N D S\otimes D S\ge K
\end{equation}
in a well-defined, weak sense on the $L^2$-Wasserstein space ${\mathcal P}_2(X)$.

We will analyze  the behavior of the entropic curvature-dimension condition $\CD^e(\KK,N)$ under transformation of the data: changing the measure $m$ into the weighted measure $m'=e^v\,m$ and   changing the length metric $d$ into
 the weighted length metric $d'$ given by
\begin{equation}
d'(x,y)
=\inf\Big\{ \int_0^1 |\dot\gamma_t|\cdot e^{w(\gamma_t)}\,dt: \ \gamma:[0,1]\to X \mbox{ rectifiable, } \gamma_0=x, \gamma_1=y\Big\}.
\end{equation}
We will not treat this problem in full generality but assume that the mms $(X,d,m)$ is {`smooth'} and also that the weights $v$ and $w$ are `smooth'.

\begin{theorem}
 If $(X,d,m)$ satisfies the  $\CD^e(\KK,N)$-condition  then for each  $N'>N$ the metric measure space $(X,d',m')$  satisfies the  $\CD^e(\KK',N')$-condition
 with a number $\KK'$ explicitly given in terms of $\K$, $v$ and $w$.
 For instance, if $v=0$
 \begin{eqnarray*}
\KK'= \inf_X \, e^{-2w}\Big[\KK-
{\LL}w+2\Gamma(w)-
\sup_{u\in\A}
\frac1{\Gamma(u)}\Big( \big(\frac{N'\,N}{N'-N}+2\big)\Gamma(w,u)^2-2 {\HH}_{w}(u,u)\Big)\Big].
\end{eqnarray*}
 If $w=0$ also $N=N'=\infty$ is admissible; if $w=\frac1N v$ also $N'=N$ is admissible.
\end{theorem}

From the very beginning, in theoretical studies and applications of the Bakry-{\'E}mery condition and of the Lott-Sturm-Villani condition, their transformation behavior under drift transformation was well studied and widely used \cite{Bak-Eme}, \cite{Stu1,Stu2,LV}.
As far as we know, however, until now for none of the other transformations of diffusion operators a transformation result for the
Bakry-{\'E}mery condition or for the Lott-Sturm-Villani condition existed.

\tableofcontents

\section{Diffusion Operators and Ricci Tensors}

\subsection{The $\Gamma$-Operator and the Hessian}

Our setting will be the following (cf. \cite{Bak94, Led2}): ${\LL}$ will be a linear operator, defined on an algebra $\A$ of functions on a set $X$ such that ${\LL}(\A)\subset\A$.
(No topological or measurability assumptions on $X$ are requested, no measure is involved.)
In terms of these data we define the square field operator
$\Gamma(f,g)=\frac12[{\LL}(fg)-f{\LL}g-g{\LL}f]$.
We assume that  ${\LL}$ is a \emph{diffusion operator}  in the sense that
\begin{itemize}
\item $\Gamma(f,f)\ge0$ for all $f\in\A$
\item $\psi(f_1,\ldots,f_r)\in\A$ for every $r$-tuple of functions $f_1,\ldots,f_r$ in $\A$ and every $C^\infty$-function $\psi:\R^r\to\R$ vanishing at the origin and
\begin{equation}\label{diff}
{\LL}\psi(f_1,\ldots,f_r)=\sum_{i=1}^r \psi_i(f_1,\ldots,f_r)\cdot {\LL}f_i + \sum_{i,j=1}^r \psi_{ij}(f_1,\ldots,f_r)\cdot \Gamma(f_i,f_j)
\end{equation}
where $\psi_i:=\frac{\partial}{\partial y_i}\psi$ and $\psi_{ij}:=\frac{\partial^2}{\partial y_i\,\partial y_j}\psi$.
\end{itemize}
We
define the Hessian of $f$ at a point $x\in X$ as a bilinear form on $\A$ by
$${\HH}_f(g,h)(x)=\frac12\Big[\Gamma\big(g,\Gamma(f,h)\big)+\Gamma\big(h,\Gamma(f,g)\big)-\Gamma\big(f,\Gamma(g,h)\big)\Big](x)$$
for $f,g,h\in\A$.
We can  always extend the definition of ${\LL}$ and $\Gamma$ to the algebra  generated by the elements in $\A$ and the constant functions which leads to ${\LL}1=0$ and $\Gamma(1,f)=0$ for all $f$.
For later use, let us state the chain rule  for ${\LL}f$ and ${\HH}_f$:
\begin{eqnarray*}\frac1p f^{-p}{\LL}f^p={\LL} \log f+p \Gamma(\log f),\qquad
\frac1p f^{-p}{\HH}_{f^p}(u,u)= {\HH}_{ \log f}(u,u)+p \Gamma(\log f,u)^2
\end{eqnarray*}
for all $p\not=0$ and all $f\in\A$ with $\log f$ and $f^p\in\A$.

\subsection{The $\Gamma_2$-Operator}

Of particular importance is the $\Gamma_2$-operator defined via iteration of the $\Gamma$-operator
$$\Gamma_2(f,g)=\frac12[{\LL}\Gamma(f,g)-\Gamma(f,{\LL}g)-\Gamma(g,{\LL}f)].$$
 We put $\Gamma(f)=\Gamma(f,f), \Gamma_2(f)=\Gamma_2(f,f)$.

\begin{lemma}[\cite{Bak94},\cite{Led2}]\label{chain-rule} Given $f_1,\ldots,f_r$ in $\A$ and a $C^\infty$-function $\psi:\R^r\to\R$.
Then
$$\Gamma(\psi(f_1,\ldots,f_r))=\sum_{i,j=1}^r [\psi_i\cdot\psi_j](f_1,\ldots,f_r)\cdot \Gamma(f_i,f_j)$$ and
\begin{eqnarray*}
\Gamma_2(\psi(f_1,\ldots,f_r))&=&\sum_{i,j}[\psi_i\cdot \psi_j](f_1,\ldots,f_r)\cdot \Gamma_2(f_i,f_j)\\
&+& 2\sum_{i,j,k}[\psi_i\cdot \psi_{jk}](f_1,\ldots,f_r)\cdot {\HH}_{f_i}(f_j,f_k)\\
&
+&\sum_{i,j,k,l}[\psi_{ij}\cdot \psi_{kl}](f_1,\ldots,f_r)\cdot \Gamma(f_i,f_k)\cdot\Gamma(f_j,f_l).
\end{eqnarray*}
\end{lemma}

\begin{remark}
We say that the family $\{f_1,\ldots,f_n\}$ is an $n$-dimensional \emph{normal coordinate system} at a given point $x\in X$ if
for all $i,j,k\in \{1,\ldots,n\}$
$$\Gamma(f_i,f_j)(x)=\delta_{ij}, \quad {\HH}_{f_i}(f_j,f_k)(x)=0, \quad {\LL}f_i(x)=0.$$
 Given such a system at the point $x\in X$, for each
 $C^\infty$-function $\psi:\R^n\to\R$ 
 and $f=(f_1,\ldots,f_n)$
 we have
 \begin{eqnarray}\label{bochner-formula}
\Gamma_2(\psi\circ f)(x) =
\Ric^\sharp(D\psi,D\psi)(f(x)) + \| D^2{\psi}\|^2_{HS}(f(x))
\end{eqnarray}
where
$\Ric^\sharp(D\psi,D\psi)(f):=\sum_{i,j=1}^n \psi_i(f)\, \psi_j(f)\, \Gamma_2(f_i,f_j)$
  and $\|D^2{\psi}\|^2_{HS}:=\sum_{i,j=1}^n |\frac{\partial^2}{\partial y_i\,\partial y_j}\psi|^2$.
Note that
 \begin{eqnarray*}\| D^2{\psi}\|^2_{HS}(f(x))&\ge&\frac1n \Big(\sum_{i=1}^n\frac{\partial^2}{\partial y_i^2}\psi\Big)^2\big(f(x)\big)\ =\ \frac1n\Big({\LL}\big(\psi\circ f\big)\Big)^2(x).
\end{eqnarray*}
\end{remark}

\subsection{The Ricci Tensor}

In terms of the $\Gamma_2$-operator we define the \emph{Ricci tensor} at the point $x\in X$ by
\begin{eqnarray}
\RR(f)(x)=\inf\big\{ \Gamma_2(\tilde f)(x): \ \tilde f\in\A, \ \Gamma(\tilde f-f)(x)=0\big\}
\end{eqnarray}
for $f\in\A$.
More generally, given any extended number $N\in[1,\infty]$ we define the \emph{$N$-Ricci tensor} at $x$ by
\begin{eqnarray}
\RR_N(f)(x)=\inf\big\{ \Gamma_2(\tilde f)(x)-\frac1N(\LL\tilde f)^2(x): \  \tilde f\in\A, \ \Gamma(\tilde f-f)(x)=0\big\}.
\end{eqnarray}
Obviously, for $N=\infty$ this yields the previously defined Ricci tensor.
Moreover, $\Gamma_2(f)\ge \RR_N(f)+\frac1N (\LL f)^2$ and $\RR(f)\ge \RR_N(f)$
 for all $N$.
One might be tempted to believe that $\RR_N(f)\ge\RR(f)-\frac1N (\LL f)^2$ but this is not true in general, see
Proposition \ref{mani-ric} below.

\begin{lemma}
For every $N\in[1,\infty]$ and every $x\in X$ the $N$-Ricci tensor is a quadratic form on $\A$.
Thus by polarization it extends to a bilinear form $\RR_N(.,.)(x)$ on its \emph{domain}
$\Dom\big(\RR_N(x)\big)=\{f\in\A: \, \RR_Nf(x)>-\infty\}$.
\end{lemma}

\begin{proof}
It suffices to prove the parallelogram inequality. Given $x\in X$, $f,g\in\Dom\big(\RR_N(x)\big)$  and $\epsilon>0$, choose
$\tilde f, \tilde g\in\A$ with $\Gamma(f-\tilde f)(x)=\Gamma(g-\tilde g)(x)=0$ such that
$\RR_N(f)(x)\ge \RR_N^0(\tilde f)(x)-\epsilon$ and $\RR_N(g)(x)\ge \RR_N^0(\tilde g)(x)-\epsilon$. Here
we put
$$\RR_N^0(h)(x)=\Gamma_2(h)(x)-\frac1N(\LL h)^2(x)$$
which obviously defines a quadratic form in $h\in\A$. Thus
 \begin{eqnarray*}
 \RR_N(f+g)(x)+\RR_N(f-g)(x)&\le&
  \RR_N^0(\tilde f+\tilde g)(x)+\RR_N^0(\tilde f-\tilde g)(x)\\
  & =&
   2\RR_N^0(\tilde f)(x)+2\RR_N^0(\tilde g)(x)\\
   &\le&
    2\RR_N( f)(x)+2\RR_N( g)(x)+4\epsilon.
 \end{eqnarray*}
 Since $\epsilon>0$ was arbitrary, this proves the claim.
\end{proof}

 Let us illustrate the concept of the $N$-Ricci tensor with two basic examples.

\begin{proposition}\label{mani-ric} Let $L$ be the Laplace-Beltrami operator on a complete $n$-dimensional Riemannian manifold (and let $\A$ be the set of ${\mathcal C}^\infty$-functions vanishing at infinity or the set of ${\mathcal C}^\infty$-functions with compact supports -- this makes no difference as long as no measure is involved). Then for all $f\in\A$
\begin{equation} \RR_N(f)=\left\{
\begin{array}{ll}
\Ric(\nabla f,\nabla f)& \mbox{if }N\ge n,\\
-\infty &\mbox{if }N<n.
\end{array}\right.
 \end{equation}
\end{proposition}

\begin{proof}
The well-known Bochner's formula states that
\begin{equation}\label{boch-form}\Gamma_2(f)=\Ric(\nabla f,\nabla f)+\big\| D^2 f\big\|^2_{HS}
\end{equation}
(cf. formula \eqref{bochner-formula}) and
 Bochner's inequality states that
 \begin{equation*}\Gamma_2(f)\ge\Ric(\nabla f,\nabla f)+\frac1n (\Delta f)^2.
\end{equation*}
Given $f$ and $x$, applying the latter to $\tilde f\in\A$ with $\nabla f(x)=\nabla\tilde f(x)$ yields for all $N\ge n$
$$\Gamma_2(\tilde f)(x)-\frac1N (\Delta \tilde f)^2(x)\ge \Ric(\nabla \tilde f,\nabla \tilde f)(x)=\Ric(\nabla f,\nabla f)(x)$$
and thus
$$\RR_Nf(x)\ge \Ric(\nabla f,\nabla f)(x).$$
Conversely, given $f$ and $x$ choose $f_0$ with $\nabla f_0(x)=\nabla f(x)$ and $D^2{f_0}(x)=0$.
Then \eqref{boch-form} implies
$$\Gamma_2(f_0)(x)-\frac1n (\Delta f_0)^2(x) =\Ric(\nabla f_0,\nabla f_0)(x)=\Ric(\nabla f,\nabla f)(x)$$
and therefore $\RR_N(f)(x)\le \Ric(\nabla f,\nabla f)(x)$.

To verify the second claim, for given $x\in\A$ and $f\in\A$, consider the functions
$f_j=f_0+j\cdot v$ with $f_0$ as above and with $v=\frac12 d(x,.)^2$. (More precisely, choose $v\in\A$ with $v=\frac12 d(x,.)^2$ in a neighborhood of $x$.)
Note that $\nabla v(x)=0$, $\Delta v(x)=n$ and $\big\| D^2 v\big\|^2_{HS}(x)=n$.
Thus  $\nabla f(x)=\nabla f_j(x)=0$ and
$$\Gamma_2(f_j)(x)-\frac1N(\Delta f_j)^2(x)=\Ric(\nabla f,\nabla f)(x)+j^2\,n-\frac{j^2\, n^2}N$$
for every $j\in\N$.
As $j\to\infty$ this proves $\RR_N(f)(x)=-\infty$ whenever $N<n$.
\end{proof}

\begin{proposition}\label{drift-ric} Let $\LL=\Delta+Z$ be the Laplace-Beltrami operator with drift on a complete $n$-dimensional Riemannian manifold where $Z$ is any smooth vector field.  Then
 \begin{equation}\RR_\infty(f)=\Ric(\nabla f,\nabla f)-(D Z)(\nabla f,\nabla f)
\end{equation}
for all $f\in\A$ and
\begin{equation} \RR_N(f)(x)=\left\{
\begin{array}{ll}
\RR_\infty(f)(x)
-\frac1{N-n}\big|Z\, f\big|^2(x)\quad&\mbox{if } N>n,\\
\RR_\infty(f)(x)&
\mbox{if }N=n \mbox{ and }(Z\, f)(x)=0,\\
-\infty&\mbox{if } N=n\mbox{ and }(Z\, f)(x)\not=0,\\
-\infty& \mbox{if }N<n.
\end{array}\right.
 \end{equation}
\end{proposition}

 \begin{proof}
 The lower estimates for $\RR_N$ will follow from our general estimates in Theorem \ref{g2-est} (cf. also Corollary \ref{drift-be}).
 For the upper estimate in the case $N<n$, for given $f$ and $x$ choose $f_j=f_0+j\cdot v$ as in the proof of the previous  proposition.
 Then
 $$\big\| D^2 f_j\big\|^2_{HS}-\frac1N(\LL f_j)^2(x)=j^2\,n-\frac{1}N\big(j\,n+(Z\,f)(x)\big)^2\to-\infty$$
as $j\to\infty$. Similarly, if $N=n$ and $(Zf)(x)\not=0$ we choose $f_j=f_0+j\cdot(Zf)(x)\cdot v$ to obtain
 $$\big\| D^2 f_j\big\|^2_{HS}-\frac1N(\LL f_j)^2(x)=j^2\,n(Zf)^2(x)-\frac{1}N(jn+1)^2(Z\,f)^2(x)\to-\infty$$
as $j\to\infty$.

If $N>n$ choose $\tilde f=f_0+\frac1{N-n}(Zf)(x)\cdot v$. Then $\nabla\tilde f(x)=\nabla f(x)$, $\Delta \tilde f(x)=\frac{n}{N-n}(Zf)(x)$ and
$\LL \tilde f(x)=\frac N{N-n}(Zf)(x)$. Moreover, $\big\| D^2 \tilde f\big\|^2_{HS}(x)=\frac n{(N-n)^2}(Zf)^2(x)$.
Thus
\begin{eqnarray*}
\RR_N(f)(x)&\le& \Gamma_2(\tilde f)-\frac1N (\LL \tilde f)^2(x)\\
&=&\Ric(\nabla f,\nabla f)(x)+(D Z)(\nabla f,\nabla f)(x)+\big\| D^2 \tilde f\big\|^2_{HS}(x)-\frac1N (\LL \tilde f)^2(x)\\
&=&\Ric(\nabla f,\nabla f)(x)+(D Z)(\nabla f,\nabla f)(x)-\frac1{N-n}(Zf)^2(x).
\end{eqnarray*}
Essentially the same argument works in the case $N=n$ and $(Zf)(x)=0$ if we put $\tilde f=f_0$.
 \end{proof}

\begin{definition}
Given a function $\K: X\to\R$ and an extended number $N\in[1,\infty]$
we say that the operator $({\LL},\A)$  satisfies the
\emph{Bakry-{\'E}mery condition} $\BE(\K,N)$ if
\begin{eqnarray}
  \Gamma_2(f)(x)\ge \K(x)\cdot\Gamma(f)(x)+\frac1N (\LL f)^2(x)
 \end{eqnarray}
 for all $f\in\A$ and all $x\in X$.
\end{definition}
The Bakry-{\'E}mery condition obviously is equivalent to the condition
 \begin{eqnarray}\RR_N(f)(x)\ge \K(x)\cdot\Gamma(f)(x) \end{eqnarray}
for all $f\in\A$ and all $x\in X$.

\section{Fundamental Estimates for $N$-Ricci Tensors and Hessians}

In the classical Riemannian setting
$\Gamma_2(f)=\Ric(\nabla f,\nabla f)+ \| D^2 f\|^2_{HS}$.
The Bakry-{\'E}mery condition $\BE(\K,n)$ relies on the bound for the Ricci tensor $\Ric(\nabla f,\nabla f)\ge k\cdot|\nabla f|^2$ and on the estimate for the Hilbert-Schmidt norm of the Hessian
$\| D^2 f\|^2_{HS}\ge\frac1n (\Delta f)^2$. A more refined analysis might be based on the identity
$$\| D^2 f\|^2_{HS}=\frac 1n (\Delta f)^2+\| D^2 f-\frac1n \Delta f\cdot \mathbb{I} \|^2_{HS}$$
and on estimates for the bilinear form (`traceless Hessian')
$$\Big( D^2 f-\frac1n \Delta f\cdot \mathbb{I}\Big)(g,h)=D^2 f (\nabla g, \nabla h)- \frac1n \Delta f \cdot \nabla g\,\nabla h.$$
For instance, one might use the estimate
$$\|a\|_{HS}^2\ge \frac n{n-1} \, \|a\|_{2}^2$$
valid for any traceless, symmetric $(n\times n)$-matrix $a=(a_{ij})_{1\le i,j\le n}$.

In the abstract setting,
 we  are now going to prove a fundamental estimate for the
 $N$-Ricci tensor in terms of
 the bilinear form
$(g,h)\mapsto {\HH}_f(g,h)-\frac1N \Gamma(g,h)\cdot {\LL}f$.
Recall that by definition $\Gamma_2(f)\ge\RR_N(f)+\frac1N (\LL f)^2$.

\begin{theorem}\label{sharpGamma2} For all $N\in[1,\infty]$ and all $f,g,h\in\A$
\begin{equation}
\Gamma_2(f)\ge\RR_N(f)+\frac1N (\LL f)^2+2
\frac{\Big[ {\HH}_f(g,h)-\frac1N \Gamma(g,h)\cdot {\LL}f\Big]^2}{ \frac{N-2}N\Gamma(g,h)^2+\Gamma(g)\cdot \Gamma(h)}.
\end{equation}
\end{theorem}
Note that $\frac{N-2}N\Gamma(g,h)^2+ \Gamma(g)\cdot\Gamma(h)\ge0$ for any choice of $g,h$ and $N\ge1$
(and that $\ldots=0$ is excluded if $N>1$ and $\Gamma(g)\cdot \Gamma(h)\ne 0$).

The previous estimate allows for two interpretations or applications. Firstly, it provides a very precise and powerful estimate for the Hessian.
We will use it in the following form.

\begin{corollary}\label{Hess-est} For all $N\in[1,\infty]$ and for all $f,g,h\in\A$
\begin{equation*}
2\Big[ {\HH}_f(g,h)-\frac1N \Gamma(g,h)\cdot {\LL}f\Big]^2\le  \Big[\Gamma_2(f)-\frac1N(\LL f)^2-\RR_N(f)\Big]\cdot \Big[ \frac{N-2}N\Gamma(g,h)^2+\Gamma(g)\cdot \Gamma(h)\Big].
\end{equation*}
and thus for each function $\rho: X\to\R$
\begin{eqnarray*}2\rho\Big|{\HH}_f(g,h)-\frac1N \Gamma(g,h)\cdot {\LL}f\Big|&\le&  \rho^2\cdot  \Big[\Gamma_2(f)-\frac1N(\LL f)^2-\RR_N(f)\Big] \\ &&+\frac12  \Big[ \frac{N-2}N\Gamma(g,h)^2+\Gamma(g)\cdot \Gamma(h)\Big].
\end{eqnarray*}
\end{corollary}\label{abs-cont-hess}
 This will be the key ingredient for the estimate of the Ricci tensor for a transformed operator.
Moreover, it implies that the Hessian $\HH_f(.)(x)$ is well-defined on equivalence classes of functions w.r.t. vanishing $\Gamma(.)(x)$.
\begin{corollary} For all $x\in X$, all $f\in\A(x)$ and all $g,h\in\A$
$$\Gamma(g)(x)=0\quad\Rightarrow\quad{\HH}_f(g,h)(x)=0.$$
\end{corollary}

  Secondly, the estimate of the theorem leads to the following refined estimate for the Ricci tensor as a consequence of which one obtains a \emph{self-improvement property} of the Bakry-{\'E}mery condition $\BE(\K,N)$.

\begin{corollary} For all $N\in[1,\infty]$ and all $f\in\A$
\begin{eqnarray}
\Gamma_2(f)&\ge&\RR_N(f)+\frac1N (\LL f)^2+
\frac N{N-1}\, \big\| {\HH}_f(.)-\frac1N {\LL}f\cdot\Gamma(.)\big\|_{\Gamma}^2\\
&\ge&
\RR_N(f)+\frac1N (\LL f)^2+
 \frac N{N-1}\Big[ \frac1{\Gamma(f)}  {\HH}_f(f,f)-\frac1N  {\LL}f\Big]^2
\end{eqnarray}
where $\|B\|_{\Gamma}(x)=\sup\{|B(g,g)|(x):\ g\in\A, \Gamma(g)(x)\le 1\}$ denotes the norm of a bilinear form $B(.)(x)$ on $\A$ w.r.t. the seminorm $\Gamma(.)(x)$.
In particular, for $N=\infty$
\begin{eqnarray}
\Gamma_2(f)&\ge&\RR(f)+ \big\| {\HH}_f(.)\big\|_{\Gamma}^2.
\end{eqnarray}
\end{corollary}

This second aspect will be taken up (and further developed) in the subsequent Chapters 4 and 5. The first aspect will be the key ingredient for the results on Ricci tensors for transformed operators in Chapters 6-10. These results will only rely on Theorem \ref{sharpGamma2}
and not on the more sophisticated results of the next two chapters.

\begin{proof}[Proof of the theorem]
We will consider functions of the form ${\tilde f}=\psi(f,g,h)\in\A$ for  smooth  $\psi:\R^3\to\R$ and use the fact that
$\Gamma_2(\tilde f)\ge\RR_N(\tilde f)+\frac1N (\LL \tilde f)^2$.
 At each point $x\in X$ we choose the optimal $\psi$.
Indeed, it suffices to consider
 ${\tilde f}=f+t[gh-g(x)h-h(x)g]$ and to optimize  in $t\in\R$ (for each given $x\in X$).

More precisely, we use equation \eqref{diff} and the assertions of Lemma \ref{chain-rule} with
$r=3$, $f_1=f$, $f_2=g$ $f_3=h$. For given $x\in X$ and $t\in\R$ we choose $\psi$ such that $\psi_1=1$, $\psi_{23}=t$ and
$\psi_2=\psi_3=\psi_{11}=\psi_{22}=\psi_{33}=\psi_{12}=\psi_{13}=0$ at the particular point $(f(x),g(x),h(x))\in\R^3$.
Then at the point $x\in X$
\begin{eqnarray*}
&&\Gamma(\tilde f-f)=0,\quad {\LL}\tilde f={\LL}f+2t\Gamma(g,h),\\
&&\Gamma_2(\tilde f)=\Gamma_2(f)+4t {\HH}_f(g,h)+2t^2\Big[\Gamma(g,h)^2+\Gamma(g)\cdot\Gamma(h)\Big].
\end{eqnarray*}
Thus $\RR_N(\tilde f)=\RR_N(f)$ and
\begin{eqnarray*}
0&\le&
\Gamma_2(\tilde f)- \RR_N(\tilde f)-\frac1N ({\LL}\tilde f)^2\\
&=& \Gamma_2(f)- \RR_N(f)-\frac1N ({\LL}f)^2\\
&&+ 4t \Big[ {\HH}_f(g,h)-\frac1N {\LL}f\cdot \Gamma(g,h)\Big]\\
&&+2t^2\Big[ \frac{N-2}N\Gamma(g,h)^2+ \Gamma(g)\cdot\Gamma(h)\Big]\\
&=:&\Gamma_2(f)- \RR_N(f)-\frac1N ({\LL}f)^2+4tb+2t^2a
\end{eqnarray*}
for $a,b$ defined by the terms in brackets.
Choosing $t=-\frac b{a}$
yields
$$0\le \Gamma_2(f)- \RR_N(f)-\frac1N ({\LL}f)^2-2\frac{b^2}{a}$$
(at the given point $x\in X$).
This is the claim.
\end{proof}


\section{Self-Improvement Property of $\Gamma_2$-Estimates}

Given $x\in X$, we denote by $\A_x^1$  the space of equivalence classes in $\A$ w.r.t. the seminorm $\Gamma(.)(x)$ and by
$\dim\, (\A,\Gamma)(x)$ the dimension of the inner product space $(\A_x^1,\Gamma(.)(x))$.
For a bilinear form $B$ on $\A_x^1$ we define its
\emph{operator norm} by $\|B\|_\Gamma(x)=\sup\{|B(v,v)|:\ v\in\A_x^1, \Gamma(v)(x)\le1\}$ and its
\emph{Hilbert-Schmidt norm} by
\begin{equation*}\big\|B\big\|_{HS}(x)=\sup\Big\{
\Big(\sum_{i,j=1}^r B(e_i,e_j)^2\Big)^{1/2}:\, r\in\N, e_1,\ldots,e_r\in\A_x^1, \Gamma(e_i,e_j)(x)=\delta_{ij}\Big\}
\end{equation*}
provided $\dim(\A,\Gamma)(x)>0$.
If $\dim(\A,\Gamma)(x)=0$ we put $\big\|B\big\|_{HS}(x)=0$. Obviously, in any case
$\|B\|_\Gamma(x)\le \big\|B\big\|_{HS}(x)$.

For any $(n\times n)$-matrix $B$ we put $\tr B=\sum_{i=1}^n B_{ii}$ and $B^o_{ij}= B_{ij}-\frac1n\tr B\,\delta_{ij}$.
Note that
$\sum_{i,j} (B_{ij})^2=\sum_{i,j} (B_{ij}^o)^2+ \frac1n (\tr B)^2$.
Similar definitions and results apply for any bilinear form on $\A_x^1$ provided $\dim\, (\A,\Gamma)(x)=n$.

\begin{theorem}\label{super-self}
\begin{itemize}
\item[(i)]
 For all $f\in\A$
\begin{equation}\label{super-eins}
\Gamma_2(f)\ge \RR(f)+\big\| \HH_f \big\|^2_{HS}.
\end{equation}
In particular, $\RR(f)(x)=-\infty$ if $\big\| \HH_f \big\|_{HS}(x)=+\infty$.
\item[(ii)]
 Moreover, for $x\in X$ and $N\in[1,\infty)$ with $N\ge n(x):=\dim\, (\A,\Gamma)(x)$
\begin{eqnarray}
\Gamma_2(f)(x)&\ge&
\RR_N(f)(x)+\big\| \HH_f(.) \big\|^2_{HS}(x)+
\frac1{N-n(x)}\Big( \tr\, \HH_f(.) - \, \LL f\Big)^2(x)\label{super-zwei0}\\
&=&
\RR_N(f)(x)+\frac1N \big(\LL f\big)^2(x)+
\big\| \HH_f(.)-\frac1N\, \LL f\cdot\Gamma(.) \big\|^2_{HS}(x)\nonumber\\
&&\qquad\qquad\qquad+
\frac1{N-n(x)}\Big( \tr\, \HH_f(.) - \frac nN\, \LL f\Big)^2(x).\label{super-zwei}
\end{eqnarray}
In the case $N=n(x)$, the respective last terms on the RHS here should be understood as the limit $N\searrow n(x)$ (which is either $0$ or $+\infty$).
\item[(iii)]
 Finally,
\begin{equation*}
\RR_N(f)(x)=-\infty
\end{equation*}
whenever $N<\dim\, (\A,\Gamma)(x)$ or if $N=\dim\, (\A,\Gamma)(x)$ and $\tr \HH_f(.)(x)\not= \LL f(x)$.
\end{itemize}
\end{theorem}

\begin{proof}
(i) Given $f$ and $x$, let us first consider the case  $\big\| \HH_f \big\|_{HS}(x)<\infty$. Here for any  $\epsilon>0$, choose $r\in\N$ and $e_1,\ldots,e_r\in\A$ with $\Gamma(e_i,e_j)(x)=\delta_{ij}$ and
$$\big\|\HH_f \big\|^2_{HS}(x)\le \sum_{i,j=1}^r \HH_f^2(e_i,e_j)(x)+ \epsilon.$$
Consider $\tilde f\in\A$ of the form
$\tilde f=f+\psi\circ e$ for smooth $\psi:\R^r\to\R$ with $\psi_i(e(x))=0$ for all $i$ where $e(x)=(e_1,\ldots,e_r)(x)$.
Then by the chain rule (see Lemma \ref{chain-rule}, applied to $\tilde\psi(y_0,y_1,\ldots,y_r)=y_0+\psi(y_1,\ldots,y_r)$ and $\tilde e=(f,e_1,\ldots,e_r)$), $\Gamma(\tilde f,.)(x)=\Gamma(f,.)(x)$ and
\begin{eqnarray}\label{g2-hf}
\Gamma_2\big(\tilde f\big)(x)&=&
\Gamma_2\big(f\big)(x)+ 2\sum_{i,j=1}^r \psi_{ij}\big(e(x)\big)\cdot \HH_f\big(e_i,e_j\big)(x)+ \sum_{i,j=1}^r\psi_{ij}^2\big(e(x)\big)
\nonumber\\
&=&
\Gamma_2\big(f\big)(x)+ \sum_{i,j=1}^r\Big(\HH_f\big(e_i,e_j\big)(x)+ \psi_{ij}\big(e(x)\big)\Big)^2-
  \sum_{i,j=^1}^r \HH_f^2\big(e_i,e_j\big)(x).
  \end{eqnarray}
Choosing $\psi$ such that $\psi_{ij}\big(e(x)\big)=-\HH_f\big(e_i,e_j\big)(x)$ for all $i,j$ leads to
\begin{eqnarray*}
\Gamma_2\big(\tilde f\big)(x)&=&
\Gamma_2\big(f\big)(x)-  \sum_{i,j=1}^r \HH_f^2\big(e_i,e_j\big)(x).
  \end{eqnarray*}
(For instance, the choice
$\psi(y_1,\ldots,y_n)=-\frac12 \sum_{i,j=1}^n \HH_f\big(e_i,e_j\big)(x)\cdot
\big(y_i-e_i(x)\big)\cdot\big(y_j-e_j(x)\big)$
will do the job.) Thus
  \begin{eqnarray*}
\RR\big( f\big)(x)&\le&
\Gamma_2\big(f\big)(x)-  \sum_{i,j=1}^r \HH_f^2\big(e_i,e_j\big)(x)\\
&\le&
\Gamma_2\big(f\big)(x)- \big\|\HH_f \big\|^2_{HS}(x)+\epsilon.
  \end{eqnarray*}
  Since $\epsilon>0$ was arbitrary this proves the claim.

Now let us consider the case $\big\| \HH_f \big\|_{HS}(x)=+\infty$. Then for any $C>0$ there exist 
$r\in\N$ and $e_1,\ldots,e_r\in\A$ with $\Gamma(e_i,e_j)(x)=\delta_{ij}$ and
$\sum_{i,j=1}^r \HH_f^2(e_i,e_j)(x)\ge C$.
With the same argument as before
$\RR\big( f\big)(x)\le
\Gamma_2\big(f\big)(x)-  \sum_{i,j=1}^r \HH_f^2\big(e_i,e_j\big)(x)
\le
\Gamma_2\big(f\big)(x)- C$.
Since $C>0$ was arbitrary this proves the claim.

(ii) Let $x\in X$, $f\in\A$ and $N> n(x)=\dim\, (\A,\Gamma)(x)$ be given.
Choose  $e_1,\ldots,e_n\in\A$ with $\Gamma(e_i,e_j)(x)=\delta_{ij}$. 
Again we will consider $\tilde f\in\A$ of the form
$\tilde f=f+\psi\circ e$ for smooth $\psi:\R^n\to\R$ with $\psi_i(e(x))=0$ for all $i$.
According to the  chain rule for $\LL$ (i.e. the property of being a 'diffusion operator')
$$\LL \tilde f(x)=\LL  f(x)+(\Delta \psi)\big(e(x)\big)\quad\mbox{with}\quad\Delta=\sum_{i=1}^n \frac{\partial^2}{\partial y_i^2}.$$
Thus
\begin{eqnarray}\label{eins}
\lefteqn{\Big[
\Gamma_2\big(\tilde f\big)(x)-\frac1N \big(\LL \tilde f\big)^2(x)\Big]
-
\Gamma_2\big( f\big)(x)}\nonumber\\
&=&
 2\sum_{i,j} \psi_{ij}\big(e(x)\big)\cdot \HH_f\big(e_i,e_j\big)(x)+ \sum_{i,j}\psi_{ij}^2\big(e(x)\big)
-\frac1N \Big(\LL f(x)+(\Delta\psi)(e(x))\Big)^2 \nonumber
\\
&=&
 \sum_{i,j}\Big(\HH_f\big(e_i,e_j\big)(x)+ \psi_{ij}\big(e(x)\big)\Big)^2\nonumber\\
 &&-
  \sum_{i,j} \HH_f^2\big(e_i,e_j\big)(x)- \frac1N \Big(\LL f(x)+(\Delta\psi)(e(x))\Big)^2 \nonumber
  \\
 &=&
 \sum_{i,j}\Big(\HH_f^o\big(e_i,e_j\big)(x)+ \psi_{ij}^o\big(e(x)\big)\Big)^2
 +\frac1n\Big( \tr \HH_f(.)(x)+(\Delta \psi)(e(x))\Big)^2\nonumber\\
 &&
 -
  \sum_{i,j} \HH_f^2\big(e_i,e_j\big)(x)- \frac1N \Big(\LL f(x)+(\Delta\psi)(e(x))\Big)^2.
   \end{eqnarray}
Now let us choose $\psi$ such that
\begin{equation}\psi_{ij}^o\big(e(x)\big)=-\HH_f^o\big(e_i,e_j\big)(x)\label{zwei}\end{equation}
for all $i,j=1\ldots,n$. Moreover, we may require that
\begin{equation}(\Delta\psi)\big(e(x)\big)=-\frac{N}{N-n}\big(\tr \HH_f-\frac nN\LL f\big)(x).\label{drei}\end{equation}
For instance, the choice
\begin{eqnarray*}\psi(y_1,\ldots,y_n)&=&-\frac12 \sum_{i,j=1}^n \HH_f^o\big(e_i,e_j\big)(x)\cdot
\big(y_i-e_i(x)\big)\cdot\big(y_j-e_j(x)\big)\\
&&
-\frac{N}{N-n}\big(\tr \HH_f-\frac nN\LL f\big)(x)\cdot
\frac1{2n}\sum_{i=1}^n
 \big(y_i-e_i(x)\big)^2
 \end{eqnarray*}
will do the job.
Combining \eqref{eins}, \eqref{zwei} and \eqref{drei} yields
\begin{eqnarray*}
\lefteqn{\Big[
\Gamma_2\big(\tilde f\big)(x)-\frac1N \big(\LL \tilde f\big)^2(x)\Big]
-
\Gamma_2\big( f\big)(x)}\nonumber\\
&=&
-  \sum_{i,j} \HH_f^2\big(e_i,e_j\big)(x)
  -\frac1{N-n}\big(\tr \HH_f-\LL f\big)^2(x).
  \end{eqnarray*}
  Thus
 \begin{eqnarray*}
\RR_N(f)(x)&\le&
\Gamma_2\big( f\big)(x)
-
  \sum_{i,j} \HH_f^2\big(e_i,e_j\big)(x)
  -\frac1{N-n}\big(\tr \HH_f-\LL f\big)^2(x).
  \end{eqnarray*}
  This is the first claim.
  (A similar argumentation proves the assertion in the case $N=n$.)

To see the equality \eqref{super-zwei} note that for each real symmetric $(n\times n)$-matrix $B$ and each scalar $a$
$$\big\|B\big\|^2_{HS}=\big\|B-\frac aN {\mathbb I} \big\|^2_{HS}+\frac1n \big( \tr\, B\big)^2-\frac1n \big( \tr\, B-\frac nNa\big)^2$$
(with ${\mathbb I}$ being the unit matrix, i.e. ${\mathbb I}_{jk}=\delta_{jk}$) and that for all $a,b\in\R$
$$\frac1nb^2-\frac1n(b-\frac nNa)^2+\frac1{N-n}(b-a)^2=\frac1N a^2+\frac1{N-n}(b-\frac nN a)^2.$$

 (iii)
 In the case $N<n$ we consider the sequence of functions $\psi^{(k)}$ given by
 \begin{eqnarray*}\psi^{(k)}(y_1,\ldots,y_n)=-\frac12 \sum_{i,j=1}^n \HH_f^o\big(e_i,e_j\big)(x)\cdot
\big(y_i-e_i(x)\big)\cdot\big(y_j-e_j(x)\big)
+
\frac k{2n}\sum_{i=1}^n
 \big(y_i-e_i(x)\big)^2.
 \end{eqnarray*}
 Then for each $k\in\N$, equations \eqref{eins} and \eqref{zwei} hold true  as before with $f^{(k)}:=f+\psi^{(k)}\circ e$ in the place of $\tilde f$ and $\psi^{(k)}$ in the place of $\psi$
 whereas instead of \eqref{drei} we obtain
 \begin{equation*}(\Delta\psi^{(k)})\big(e(x)\big)=k.\label{drei'}\end{equation*}
Thus
 \begin{eqnarray*}
 \RR_N(f)(x)\le \inf_{k\in\N}
 \Big[
\Gamma_2\big( f^{(k)}\big)(x)-\frac1N \big(\LL  f^{(k)}\big)^2(x)\Big]=-\infty.
 \end{eqnarray*}
\end{proof}

Theorem \ref{super-self} implies a \emph{strong self-improvement property} of the Bakry-{\'E}mery condition.

\begin{corollary}\label{be-N-imp} The Bakry-{\'E}mery condition $\BE(\K,N)$ -- that is, the condition $\Gamma_2(f)\ge\K\,\Gamma(f)+\frac1N (\LL f)^2$ for all $f\in \A$  -- implies that $N\ge\dim(\A,\Gamma)(.)$ everywhere on $X$ and that the following scale of `improved $\BE(\K,N)$-inequalities' hold true for all $f\in\A$
\begin{eqnarray}
\Gamma_2(f)
&\ge&
\K\,\Gamma(f)+\frac1N (\LL f)^2+
\big\| \HH_f(.)-\frac1N \LL f\cdot\Gamma(.)\big\|_{HS}^2+\frac1{N-n}\big(\tr\,\HH_f-\frac nN\LL f\big)^2\label{first}\\
&\ge&
\K\,\Gamma(f)+\frac1N (\LL f)^2 +\frac N{N-1} \, \big\| {\HH}_f(.)-\frac1N {\LL}f\cdot\Gamma(.)\big\|_{\Gamma}^2\label{second}
\\
&\ge&
\label{Ba-Qi}
\K\,\Gamma(f)+\frac1N (\LL f)^2 +\frac N{N-1}\Big[ \frac1{\Gamma(f)}  {\HH}_f(f,f)-\frac1N  {\LL}f\Big]^2
\end{eqnarray}
pointwise on $X$ where $n(x)=\dim(\A,\Gamma)(x)$.
\end{corollary}

\begin{proof}
It only remains to prove the step from  \eqref{first} to \eqref{second}.
This follows from
\begin{eqnarray*}
\lefteqn{\big\| \HH_f(.)-\frac1N \LL f\cdot\Gamma(.)\big\|_{HS}^2}\\
&=&
\big\| \HH_f^o(.)\big\|_{HS}^2+\frac1n \big(\tr\,\HH_f-\frac nN\LL f\big)^2\\
&\ge&\frac n{n-1}\big\| \HH_f^o(.)\big\|_{\Gamma}^2+\frac1n \big(\tr\,\HH_f-\frac nN\LL f\big)^2\\
&\ge&\frac n{n-1}\Big[\big\| \HH_f(.)-\frac1N \LL f\cdot\Gamma(.)\big\|_{\Gamma} - \frac1n \big(\tr\,\HH_f-\frac nN\LL f\big)\Big]^2+\frac1n \big(\tr\,\HH_f-\frac nN\LL f\big)^2\\
&\ge&\frac N{N-1}\big\| \HH_f(.)-\frac1N \LL f\cdot\Gamma(.)\big\|_{\Gamma}^2 - \frac1{N-n} \big(\tr\,\HH_f-\frac nN\LL f\big)^2.
\end{eqnarray*}
Here we used the fact that
$\|B\|_{HS}^2\ge \frac n{n-1} \, \|B\|_{2}^2$
 for any traceless, symmetric $(n\times n)$-matrix $B$ and that $\frac n{n-1}(a-\frac1n b)^2+(\frac 1n+\frac1{N-n})b^2\ge\frac N{N-1}a^2$ for any pair of numbers $a,b$.
\end{proof}

The last version \eqref{Ba-Qi} is the `extended $\BE(\K,N)$-inequality' derived in \cite{Bak-Qia}. See also \cite{Sav} for recent generalizations (with $N=\infty$) to metric measure spaces.
Let us reformulate the previous result for the case $N=\infty$.
\begin{corollary}\label{be-8-imp}
 The Bakry-{\'E}mery condition $\BE(\K,\infty)$ -- that is, the condition $\Gamma_2(f)\ge\K\,\Gamma(f)$ for all $f\in \A$  -- implies the following scale of `improved $\BE(\K,\infty)$-inequalities'  for all $f\in\A$
\begin{eqnarray}
\Gamma_2(f)
&\ge&
\K\,\Gamma(f)+
\big\| \HH_f(.)\big\|_{HS}^2\\
&\ge&
\K\,\Gamma(f)+ \big\| {\HH}_f(.)\big\|_{\Gamma}^2\\
&\ge&
\K\,\Gamma(f)+\Big[ \frac1{\Gamma(f)}  {\HH}_f(f,f)\Big]^2.
\end{eqnarray}
\end{corollary}

The importance of the 
inequalities \eqref{super-eins} and \eqref{super-zwei0}
-- which lead to the
very first inequality in Corollaries \ref{be-N-imp} and  \ref{be-8-imp} --
will become evident in the next chapter where  under mild assumptions we prove that these are indeed \emph{equalities}.

\section{The Bochner Formula}
For the sequel, fix $x\in X$ and a family $\{e_i:\, i\in I\}\subset\A$ which is orthonormal w.r.t. $\Gamma(.)(x)$, i.e.
$\Gamma(e_i,e_j)(x)=\delta_{ij}$. Let $I$ be either $\N$ or $\{1,\ldots,n\}$ for some $n\in\N$.
We say that the system $\{e_i:\, i\in I\}\subset\A$ is \emph{regular} if 
for all $f,g\in\A$ with $\Gamma(g)(x)=0$ there exist a sequence of smooth $\psi^r:\R^r\to\R$ with
$\frac{\partial}{\partial y_i}\psi^r(e_1,\ldots,e_r)(x)=0$ for all $i\in I$ such that
\begin{equation}\label{reg1}
\Gamma_2(f+g_r)(x)\to\Gamma_2(f+g)(x)\qquad\mbox{as }r\to\infty
\end{equation}
where $g_r=\psi^r\circ(e_1,\ldots,e_r)$ and -- if $\dim(\A,\Gamma)(x)<\infty$ -- in addition
 \begin{equation}\label{reg2}
\LL g_r(x)\to\LL g(x)\qquad\mbox{as }r\to\infty.
\end{equation}

\begin{theorem} Let $x\in X$ be given as well as a regular orthonormal system $\{e_i:\, i\in I\}$.
 \begin{itemize}
 \item[(i)] Then for all $f\in\A$
\begin{equation}
\Gamma_2(f)(x)= \RR(f)(x)+\big\| \HH_f \big\|^2_{HS}(x).
\end{equation}
\item[(ii)] Moreover, for $N\ge n(x):=\dim\, (\A,\Gamma)(x)$
\begin{eqnarray}
\Gamma_2(f)(x)
&=&\RR_N(f)(x)+\big\| \HH_f(.) \big\|^2_{HS}(x)+
\frac1{N-n(x)}\Big( \tr\, \HH_f(.) - \, \LL f\Big)^2(x)\\
&=&\RR_N(f)(x)+\frac1N \big(\LL f\big)^2(x)+
\big\| \HH_f(.)-\frac1N\, \LL f\cdot\Gamma(.) \big\|^2_{HS}(x)\nonumber
\\
&&\qquad\qquad\qquad+
\frac1{N-n(x)}\Big( \tr\, \HH_f(.) - \frac nN\, \LL f\Big)^2(x).
\end{eqnarray}
In the case $N=n(x)$, the last term on the RHS here should be understood as the limit $N\searrow n(x)$ (which is either $0$ or $+\infty$).
\item[(iii)] $\RR_N(f)(x)>-\infty$ if and only if $N>\dim\, (\A,\Gamma)(x)$ or if $N=\dim\, (\A,\Gamma)(x)$ and $\tr\,\HH_f(x)=\LL f(x)$.
\end{itemize}
\end{theorem}

\begin{proof}
(i)
Let $f,g\in\A$ be given with $\Gamma(g)(x)=0$.
Since $\{e_i:\, i\in I\}$ is regular at $x$ we may approximate $g$ by $g_r\in\A$  of the form
$ g_r=\psi^r\circ (e_1,\ldots,e_r)$ for  smooth $\psi^r:\R^r\to\R$ with $\psi^r_i(e(x))=0$ for all $i$.
(For $r>n(x)=\dim(\A,\Gamma)(x)$ the function $\psi^r$ should be a function of the first $n$ coordinates.)
Recall from \eqref{g2-hf}
\begin{eqnarray*}
\Gamma_2\big(f+g_r\big)(x)
&=&
\Gamma_2\big(f\big)(x)+ \sum_{i,j=1}^r\Big(\HH_f\big(e_i,e_j\big)(x)+ \psi_{ij}\big(e(x)\big)\Big)^2-
  \sum_{i,j=^1}^r \HH_f^2\big(e_i,e_j\big)(x)\\
  &\ge&
\Gamma_2\big(f\big)(x)-
  \sum_{i,j=^1}^r \HH_f^2\big(e_i,e_j\big)(x)\\
    &\ge&
\Gamma_2\big(f\big)(x)-
\big\| \HH_f \big\|^2_{HS}(x).  
  \end{eqnarray*}
The regularity assumption thus implies
$\Gamma_2\big(f+g\big)(x)\ge\Gamma_2\big(f\big)(x)-
\big\| \HH_f \big\|^2_{HS}(x)$ for all $g\in\A$ with $\Gamma(g)(x)=0$. 
Therefore, $\RR(f)(x)\ge  \Gamma_2\big(f\big)(x)-\big\| \HH_f \big\|^2_{HS}(x)$.
Together with the upper estimate from Theorem \ref{super-self} this proves the claim.

(ii) Now let us assume $I=\{1,\ldots,n\}$ with $n=\dim (\A,\Gamma)(x)<\infty$
and let us approximate $g$ by $g_r\in\A$  of the form
$ g_r=\psi^r\circ (e_1,\ldots,e_n)$ for  smooth $\psi^r:\R^n\to\R$ with $\psi^r_i(e(x))=0$ for all $i$.
Recall from \eqref{eins}
\begin{eqnarray*}
\lefteqn{\Big[
\Gamma_2\big(f+g_r\big)(x)-\frac1N \big(\LL (f+g_r)\big)^2(x)\Big]
-
\Gamma_2\big( f\big)(x)}\nonumber\\
 &=&
 \sum_{i,j=1}^n\Big(\HH_f^o\big(e_i,e_j\big)(x)+ (\psi_{ij}^r)^o\big(e(x)\big)\Big)^2
 +\frac1n\Big( \tr \HH_f(.)(x)+(\Delta \psi^r)(e(x))\Big)^2\nonumber\\
 &&
 -
  \sum_{i,j=1}^n \HH_f^2\big(e_i,e_j\big)(x)- \frac1N \Big(\LL f(x)+(\Delta\psi^r)(e(x))\Big)^2\\
   &\ge&
 - \ \sum_{j,k=1}^n \HH_{f}^2(e_j,e_k)(x)
 -\frac1{N-n} \Big( \tr \HH_f(.)(x)-  \LL f(x)\Big)^2.
   \end{eqnarray*}
Passing to the limit $r\to\infty$ this yields
 $$\Big[
\Gamma_2\big(f+g\big)(x)-\frac1N \big(\LL (f+g)\big)^2(x)\Big]
-
\Gamma_2\big( f\big)(x)\ge
 -  \big\| \HH_f \big\|^2_{HS}(x)
 -\frac1{N-n} \Big( \tr \HH_f(.)(x)-  \LL f(x)\Big)^2$$
 for every $g\in\A$ with $\Gamma(g)(x)=0$. In other words,
 $$\RR_N(f)(x)\ge  \Gamma_2\big(f\big)(x)-\big\| \HH_f \big\|^2_{HS}(x)-\frac1{N-n} \Big( \tr \HH_f(.)(x)-  \LL f(x)\Big)^2.$$
     Together with the upper estimate from Theorem \ref{super-self} this proves the claim.
\end{proof}
  
Let us add some brief discussion on the regularity assumption in the previous theorem. (This assumption will not be used at any other place in this paper.)

\begin{lemma}
Assume that for given $x\in X$ an orthonormal system $\{e_i:\, i\in I\}$ satisfies $\big\| \HH_f \big\|^2_{HS}(x)=\sum_{i,j} \HH_f (e_i,e_j)^2<\infty$ for all $f\in\A$ and
\begin{equation}\label{g2=hess}
\Gamma_2(f,g)(x)=\big\langle \HH_f,\HH_g\big\rangle_{HS}(x)
\end{equation}
for all $g\in\A$ with $\Gamma(g)(x)=0$.
Then $\{e_i:\, i\in I\}$ satisfies condition \eqref{reg1} in the definition of 'regularity'.
If in addition 
\begin{equation}\label{L=tr}
\LL g(x)=\tr\, \HH_g(x)
\end{equation}
for all $g\in\A$ with $\Gamma(g)(x)=0$
then $\{e_i:\, i\in I\}$ satisfies condition \eqref{reg2} in the definition of 'regularity'.
\end{lemma}

Note that \eqref{g2=hess} and \eqref{L=tr} are always satisfied if $f=\phi\circ (e_1,\ldots,e_r)$ and $g=\psi\circ (e_1,\ldots,e_r)$ for some smooth 
$\phi, \psi:\R^r\to\R$. Indeed,
Lemma \ref{chain-rule} and Lemma \ref{hess-calc} from below imply
\begin{eqnarray*}
\Gamma_2(f,g)(x)&=&\sum_{i,j,k}(\phi_i\,\psi_{jk})(e(x))\cdot \HH_{e_i}(e_j,e_k)(x)+\sum_{j,k}(\phi_{jk}\,\psi_{jk})(e(x))\\
&=& \sum_{j,k}\HH_f(e_j,e_k)(x)\cdot \HH_g(e_j,e_k)(x)
\end{eqnarray*}
and
$$\LL g(x)=\sum_{i}\psi_{ii}(e(x))=\tr\,\HH_f(x).$$

\begin{proof}[Proof of the Lemma]
First, observe that \eqref{g2=hess} implies
$\Gamma_2(g)(x)=\big\| \HH_g\big\|^2_{HS}(x)$ and thus
\begin{equation}
\Gamma_2(f+g)(x)=\Gamma_2(f)(x)+2\big\langle \HH_f,\HH_g\big\rangle_{HS}(x)+\big\| \HH_g\big\|_{HS}(x).
\end{equation}
Put $\psi^r(y_1,\ldots,y_r)=\frac12\sum_{i,j=1}^r \HH_g(e_j,e_k)(x)\cdot\Big(\big[y_j-e_j(x)\big]\cdot\big[y_k-e_k(x)\big]-e_j(x)\cdot e_k(x)
\Big)
$ and $g_r=\psi^r\circ(e_1,\ldots,e_r)$.
Then according to the chain rule
\begin{eqnarray*}
\Gamma_2(f+g_r)(x)&=&
\Gamma_2(f)(x)+2\sum_{j,k=1}^r  \HH_f(e_j,e_k)(x)\cdot\HH_g(e_j,e_k)(x)
+\sum_{j,k=1}^r  \HH_g(e_j,e_k)^2(x)\\
&\to
&\Gamma_2(f)(x)+2\big\langle \HH_f,\HH_g\big\rangle_{HS}(x)+\big\| \HH_g\big\|_{HS}(x)\ =\ \Gamma_2(f+g)(x)
\end{eqnarray*}
as $r\to\infty$.

In the  case $n=\dim(\A,\Gamma)(x)<\infty$ put $\psi^r$ and $g_r=\psi^r\circ (e_1,\ldots,e_n)$ as above with $r:=n$. 
Then $\LL g_r(x)=(\Delta\psi^r)(e(x))$ and $\Delta\psi^r(.)=\sum_{i=1}^r\HH_g(e_i,e_i)(x)$ uniformly on $\R^r$. Thus
$\LL g_r(x)=\tr\HH_g(x)$.   
\end{proof}

\begin{remark}
Given  any  family $\{e_i:\, i\in I\}\subset\A$ which is orthonormal w.r.t. $\Gamma(.)(x)$  we put
$$\eta_{jk}(.)=\frac12\big[e_j-e_j(x)\big]\cdot \big[e_k-e_k(x)\big]-\frac12 e_j(x)e_k(x)\in\A$$
and $I_2=\{(j,k)\in I^2:\, j\le k\}$.
Then the family $\{\hat\eta_{jk}\}_{(j,k)\in I_2}$ with $\hat\eta_{jj}=\eta_{jj}$ and $\hat\eta_{jk}=\sqrt2\, \eta_{jk}$ if $j<k$ is orthonormal w.r.t. $\Gamma_2(.)(x)$.
Indeed, Lemma \ref{chain-rule} implies via polarization that
$$\Gamma_2\big(\phi\circ e,\psi\circ e\big)(x)=\sum_{j,k}\big(\phi_{jk}\cdot\psi_{jk}\big)\big(e(x)\big)$$
for all smooth $\phi,\psi:\R^n\to\R$ with $\phi_i(e(x))=\psi_i(e(x))=0$ for all $i$.
\end{remark}

\begin{proposition}
Given $x\in X$ and an orthonormal system $\{e_i:\, i\in I\}$  w.r.t. $\Gamma(.)(x)$. 
Condition \eqref{g2=hess} implies that the family $\{\hat\eta_{jk}\}_{(j,k)\in I_2}$ is a \emph{complete} orthonormal system
for the pre-Hilbert space $(\A_x^2,\Gamma_2(.)(x))$
where $\A_x^2$ denotes the set of equivalence classes in $\{f\in\A: \ \Gamma(f)(x)=0\}$ w.r.t. the relation
$f\approx g\Leftrightarrow \Gamma_2(f-g)(x)=0$. 

More precisely, for any  orthonormal system $\{e_i:\, i\in I\}$
the following assertions are equivalent:
\begin{itemize}
\item[(i)]  For all $f,g\in\A^2_x$
\begin{equation}\label{hess=g2}
\Gamma_2(f,g)(x)=\big\langle\HH_f, \HH_g\big\rangle_{HS}(x).
\end{equation}
\item[(ii)]
For all $f\in\A^2_x$
\begin{equation}\label{bi-cpl}
\Gamma_2(f)(x)=\sum_{j,k\in I}\Gamma_2(\eta_{jk},f)^2(x)
\end{equation}
and for all $j,k\in I$
\begin{equation}\label{hessjk}
\Gamma_2(f,\eta_{jk})(x)=\HH_f(e_j,e_k)(x).
\end{equation}
\end{itemize}
\end{proposition}

\begin{proof} For the function $g=\eta_{jk}$, the subsequent lemma implies $\HH_g(e_i,e_l)(x)=\frac12\delta_{ij}\cdot\delta_{lk}+\frac12\delta_{ik}\cdot\delta_{lj}$.
Assumption \eqref{hess=g2} thus implies \eqref{hessjk}.

On the other hand, assuming \eqref{hessjk} obviously  yields the equivalence of \eqref{hess=g2} and \eqref{bi-cpl}.
\end{proof}

\begin{lemma}\label{hess-calc}
Assume that $f=\psi\circ (e_1,\ldots,e_r)$ for some smooth $\psi:\R^r\to\R$ and an orthonormal system $\{e_1,\ldots,e_r\}$. Then
 \begin{eqnarray}\label{hess-chain}
  \HH_f(e_j,e_k)(x)=\sum_i\psi_i(e(x))\, \HH_{e_i}(e_j,e_k)(x)+\psi_{jk}(e(x)).
   \end{eqnarray}
\end{lemma}

\begin{proof}
Note that e.g.
$\Gamma(e_j,\Gamma(f,e_k))=\sum_i \psi_i(e)\,\Gamma(e_j,\Gamma(e_i,e_k))+\sum_{i,l}\psi_{il}(e)\,\Gamma(e_l,e_k)\,\Gamma(e_i,e_j)$
and thus
 \begin{eqnarray*}
 2 \HH_f(e_j,e_k)&=&\Gamma(e_j,\Gamma(f,e_k))+\Gamma(e_k,\Gamma(f,e_j))-\Gamma(f,\Gamma(e_j,e_k))\\
 &=&2\sum_i\psi_i(e) \, \HH_{e_i}(e_j,e_k)+ 2\sum_{i,l}\psi_{il}(e)\,\Gamma(e_l,e_k)\,\Gamma(e_i,e_j)
   \end{eqnarray*}
everywhere on $X$. Using the orthonormality of the $e_j$ at $x$ yields \eqref{hess-chain}.
\end{proof}

Let us conclude this chapter with an example illustrating that the dimension 
$\dim\, (\A,\Gamma)(.)$
might be non-constant on $X$

\begin{example}
Let $X=\R^2$, $\A={\mathcal C}_c^\infty(X)$ and
$$\LL f(x)=\phi(x_1)\,\frac{\partial^2}{\partial x_1^2}f(x)
+\frac12\phi'(x_1)\,\frac{\partial}{\partial x_1}f(x)
+\frac{\partial^2}{\partial x_2^2}f(x)$$
for some ${\mathcal C}^\infty$-function $\phi:\R\to\R$ with $\phi>0$ on $(0,\infty)$ and $\phi=0$ on $(-\infty,0]$. Then
$$ \dim\, (\A,\Gamma)(x)=\left\{
\begin{array}{ll}
1& \quad \mbox{if }x_1\le0,\\
2& \quad \mbox{if }x_1>0.
\end{array}
\right.$$
Moreover, 
$\RR_Nf(x)=0$ for all $f\in\A$ and all $N\ge2$ and
$$ \RR_N(f)(x)=\left\{
\begin{array}{ll}
0& \quad \mbox{if }x_1\le0,\\
-\infty& \quad \mbox{if }x_1>0
\end{array}
\right.$$
for  $N\in[1,2)$.  

Indeed, by construction
$$\Gamma(f)(x)= \phi(x_1)\,\Big|\frac{\partial}{\partial x_1}f\Big|^2(x)+\Big|\frac{\partial}{\partial x_2}f\Big|^2(x).$$
The assertion on $\dim\, (\A,\Gamma)(.)$ thus is obvious.
By Theorem \ref{super-self}(iii) it implies the assertion on $\RR_Nf(x)$ in the case $x_1>0$ and $N<2$. In the cases $x_1\le0$ or $N\ge2$, the assertion follows from the analogous assertion for the 1-dimensional diffusion in $x_2$-direction (which is a conformal transformation of the standard diffusion in $x_2$-direction), cf. Theorem \ref{conf-ricc}.
\end{example}

\section{Ricci Tensor for Transformed Operators}

In the sequel we will study the operator
\begin{equation}\label{Lsharp2}
{\LL}' u=f^2\, {\LL}u+f\,\sum_{i=1}^r g_i \, \Gamma(h_i,u)
\end{equation}
for given $r\in \N$ and functions $f, g_i,h_i\in\A$ (for $i=1,\ldots,r$).
Obviously, the associated $\Gamma$-operator is given by $\Gamma'(u)=f^2\, \Gamma(u)$.
Our main result is the following estimate for the $N'$-Ricci tensor for $\LL'$ in terms of the $N$-Ricci tensor for $\LL$.

\begin{theorem}\label{g2-est} For every $N'>N$
\begin{eqnarray*}
\RR'_{N'}(u)
&\ge&
f^4\, \RR_N(u)-\frac1{N'-N}\Big(\frac{N-2}2\Gamma(f^2,u)+\sum_i f g_i\Gamma(h_i,u)\Big)^2\\
&&+
\frac12\big(f^2{\LL}f^2-\Gamma(f^2)\big)\Gamma(u)-\frac{N-2}4\Gamma(f^2,u)^2-\sum_i f^3g_i {\HH}_{h_i}(u,u)\\
&&+\frac12\sum_i f g_i\Gamma(h_i,f^2)\Gamma(u)-\sum_i f^2\Gamma(f g_i,u)\Gamma(h_i,u).
\end{eqnarray*}
In the particular case $r=1$, $g_1=-(N-2)$, $h_1=f$ one may also choose $N'=N$.
\end{theorem}

\begin{corollary}\label{g2-est-cor}
Assume that the operator ${\LL}$ satisfies the $\BE(\K,N)$-condition. Then for every $N'>N$
the operator ${\LL}'$ satisfies the $\BE(\K',N')$-condition with
\begin{eqnarray*}
\K'&:=& f^2\,\K+
\frac12{\LL}f^2-2\Gamma(f)+\sum_i  g_i\Gamma(h_i,f)\\
&&+
\inf_{u\in\A}
\frac1{\Gamma(u)}\Big[-\frac1{N'-N}\Big((N-2)\Gamma(f,u)+\sum_i  g_i\Gamma(h_i,u)\Big)^2\\
&&\qquad\qquad -(N-2)\Gamma(f,u)^2-\sum_i f g_i {\HH}_{h_i}(u,u)-\sum_i \Gamma(f g_i,u)\Gamma(h_i,u)\Big].
\end{eqnarray*}
\end{corollary}

\begin{proof} All the subsequent statements will be pointwise statements for a given $x\in X$. It suffices to consider $f\in\Dom(\RR_N(x))$.
By the very definition of ${\LL}'$, $\Gamma'$ and $\Gamma_2'$ we obtain for all $u\in\A$
 \begin{eqnarray*}
\Gamma_2'(u)&=&
\frac12 f^2 {\LL}(f^2\Gamma(u))+\frac12\sum_if g_i\Gamma(h_i,f^2\Gamma(u))-f^2\Gamma(u,f^2 {\LL}u)-f^2\Gamma(u,\sum_i f g_i\Gamma(h_i,u))\\
&=&
f^4 \Gamma_2(u)+\frac12 f^2 {\LL}f^2\cdot \Gamma(u)+2f^2\Big({\HH}_u(f^2,u)-\frac1N {\LL}u\cdot\Gamma(f^2,u)\Big)-\frac{N-2}N f^2\Gamma(f^2,u){\LL}u\\
&&+\sum_i\Big(-f^3 g_i {\HH}_{h_i}(u,u)+\frac12 f g_i\Gamma(h_i,f^2)\Gamma(u)-f^2\Gamma(f g_i,u)\Gamma(h_i,u)\Big)\\
&=&
f^4 \Gamma_2(u)+2f^2\Big({\HH}_u(f^2,u)-\frac1N {\LL}u\cdot\Gamma(f^2,u)\Big) -\frac{N-2}N f^2\Gamma(f^2,u){\LL}u+A(f,g,h,u)\\
&=&
f^4\,\RR_N(u)+f^4 \big(\Gamma_2(u)-\RR_N(u)-\frac1N(\LL u)^2     \big)\\
&&+2f^2\Big[{\HH}_u(f^2,u)-\frac1N {\LL}u\cdot\Gamma(f^2,u)\Big]+\frac{N-2}{2N} \Gamma(f^2,u)^2\\
&&+\frac12\Gamma(f^2)\Gamma(u)+\frac1N\Big(f^2{\LL}u-\frac{N-2}2\Gamma(f^2,u)\Big)^2+A'(f,g,h,u)
\end{eqnarray*}
with  \begin{eqnarray*}
A(f,g,h,u):=\frac12 f^2 {\LL}f^2\,\Gamma(u)
 +\sum_i\Big(-f^3 g_i {\HH}_{h_i}(u,u)+\frac12 f g_i\Gamma(h_i,f^2)\Gamma(u)-f^2\Gamma(f g_i,u)\Gamma(h_i,u)\Big)
\end{eqnarray*}
and
\begin{eqnarray*}
A'(f,g,h,u):=A(f,g,h,u)-\frac12\Gamma(f^2)\Gamma(u)-\frac{N-2}4\Gamma(f^2,u)^2.
\end{eqnarray*}
Corollary \ref{Hess-est} provides a sharp lower estimate for the above $[\ .\ ]$-term (involving the 'traceless Hessian' of $u$) which leads to
 \begin{eqnarray*}
\Gamma_2'(u)&\ge&
f^4\,\RR_N(u)+\frac1N\Big(f^2{\LL}u-\frac{N-2}2\Gamma(f^2,u)\Big)^2+A'(f,g,h,u).
\end{eqnarray*}
Therefore,
 \begin{eqnarray*}
\Gamma_2'(u)-\frac1{N'}({\LL}'u)^2&\ge&
f^4\,\RR_N(u)-\frac1{N'}\Big(f^2{\LL}u+\sum_if g_i\Gamma(h_i,u)\Big)^2\\
&&\qquad\qquad +\frac1N\Big(f^2{\LL}u-\frac{N-2}2\Gamma(f^2,u)\Big)^2
+A'(f,g,h,u)\\
&\ge&
f^4\,\RR_N(u)-\frac1{N'-N}\Big(\frac{N-2}N \Gamma(f^2,u)+\sum_if g_i\Gamma(h_i,u)\Big)^2+A'(f,g,h,u).
\end{eqnarray*}
Given $u_0\in\A$ and varying among all $u\in\A$ with $\Gamma(u-u_0)(x)=0$ (for the given  $x\in X$) then yields
 \begin{eqnarray*}
 \RR'_{N'}(u_0)(x)&=&
 \inf_{u\in \A,\, \Gamma(u-u_0)(x)=0}\Big[\Gamma_2'(u)(x)-\frac1{N'}\big(\LL' u\big)^2(x)\Big]\\
 &\ge&
f^4\,\RR_N(u)-\frac1{N'-N}\Big(\frac{N-2}N \Gamma(f^2,u)+\sum_if g_i\Gamma(h_i,u)\Big)^2\\
&&\qquad +\inf_{u\in \A,\, \Gamma(u-u_0)(x)=0}A'(f,g,h,u).
\end{eqnarray*}
According to Corollary \ref{abs-cont-hess},  $A'(f,g,h,u)=A'(f,g,h,u_0)$  for all $u$ under consideration. This proves the claim.
\end{proof}

\section{Conformal Transformation}
The previous results significantly simplify in the case
$r=1$, $g_1=-(N-2)$, $h_1=f$.
This is the only case
where we can estimate the $N$-Ricci tensor for the transformed operator in terms of the $N$-Ricci tensor of the original one.
It is also the only case where the Bakry-{\'E}mery condition for $\LL$ will imply  a Bakry-{\'E}mery condition
for the transformed operator    with the \emph{same} dimension bound $N$.

Put $$\tilde {\LL}=f^2{\LL}-\frac{N-2}2\Gamma(f^2,u)$$ and let $\tilde \Gamma, \tilde\Gamma_2, \tilde\RR_N$ be the associated square field operator, iterated square field operator, and $N$-Ricci tensor, respectively. Theorem \ref{g2-est} immediately yields

\begin{corollary} For all $u\in\A$
\begin{eqnarray*}
\tilde\RR_N(u)
&\ge& f^4\, \RR_N(u)+
\Big(
\frac12 f^2{\LL}f^2-\frac N4\Gamma(f^2)\Big)\Gamma(u)-\frac{N-2}4\Gamma(f^2,u)^2+\frac{N-2}2 f^2 {\HH}_{f^2}(u,u).
\end{eqnarray*}
In particular, if the operator ${\LL}$ satisfies the $\BE(\K,N)$-condition then
the operator $\tilde {\LL}=f^2{\LL}-\frac{N-2}2\Gamma(f^2,u)$ satisfies the $\BE(\tilde\K,N)$-condition with
\begin{eqnarray*}
\tilde\K&:=& f^2\,\K+
\frac12{\LL}f^2-N\Gamma(f)+
\inf_{u\in\A}
\frac1{\Gamma(u)}\Big[-(N-2)\Gamma(f,u)^2+\frac{N-2}2  {\HH}_{f^2}(u,u)\Big]\\
&=& f^2\,\K+
f{\LL}f-(N-1)\Gamma(f)+
\inf_{u\in\A}
\frac{N-2}{\Gamma(u)}f\, {\HH}_{f}(u,u).
\end{eqnarray*}
\end{corollary}

If $f>0$ with $\log f\in\A$ the above estimate for the $N$-Ricci tensor indeed becomes an equality.

\begin{theorem}\label{conf-ricc} Given any $w\in\A$ and $N\in[1,\infty]$ let $\tilde\RR_N$ denote the $N$-Ricci tensor associated to the operator
$\tilde L=e^{-2w}\big(L+(N-2)\Gamma(w,.)\big)$.
Then for all $u\in\A$
\begin{eqnarray}\label{conf-est-exp}
\tilde\RR_N(u)
&=&
e^{-4w}\,\Big( \RR_N(u)+\big[-{\LL} w -(N-2)\Gamma(w)\big]\, \Gamma(u)\nonumber\\
&&\qquad\qquad -(N-2){\HH}_{w}(u,u)+(N-2)\Gamma(w,u)^2\Big).
\end{eqnarray}
\end{theorem}

\begin{proof}
Firstly, we apply the previous corollary  with
$f=e^{-w}$ to obtain a lower bound for the $N$-Ricci tensor $\tilde\RR_N$ associated to the operator $\tilde\LL=e^{-2w}(\LL +(N-2)\Gamma(w,.))$ in terms of the $N$-Ricci tensor $\RR_N$ associated to the operator $\LL$: This yields the ''$\ge$'' in \eqref{conf-est-exp}:
\begin{eqnarray}\label{conf-upp}
\tilde\RR_N(u)
&\ge&
e^{-4w}\,\Big( \RR_N(u)+\big[-{\LL} w -(N-2)\Gamma(w)\big]\, \Gamma(u)\nonumber\\
&&\qquad\qquad -(N-2){\HH}_{w}(u,u)+(N-2)\Gamma(w,u)^2\Big).
\end{eqnarray}
Secondly, we apply the previous corollary with $f=e^w$ to obtain a lower bound for the $N$-Ricci tensor associated to the operator $e^{2w}(\tilde\LL -(N-2)\tilde\Gamma(w,.))$ in terms of the $N$-Ricci tensor $\tilde\RR_N$ associated to the operator $\tilde\LL$.
Note that $e^{+2w}(\tilde\LL -(N-2)\tilde\Gamma(w,.))=\LL$. Thus this indeed provides us with a lower bound for $\RR_N$:
\begin{eqnarray}\label{conf-low}
\RR_N(u)
&\ge&
e^{+4w}\,\Big( \tilde\RR_N(u)+\big[+{\tilde\LL} w -(N-2)\tilde\Gamma(w)\big]\, \tilde\Gamma(u)\nonumber\\
&&\qquad\qquad +(N-2){\tilde\HH}_{w}(u,u)+(N-2)\tilde\Gamma(w,u)^2\Big).
\end{eqnarray}
Combining these two estimates and using the fact that
\begin{eqnarray*}
\tilde\HH_w(u,u)&=&\tilde\Gamma(u,\tilde\Gamma(u,w))-\frac12\tilde\Gamma(w,\tilde\Gamma(u))\\
&=& e^{-4w}\cdot\Big[\HH_w(u,u)-2\Gamma(w,u)^2+\Gamma(w)\cdot\Gamma(u)\Big]
\end{eqnarray*}
finally yields that \eqref{conf-upp} and \eqref{conf-low} are indeed equalities.
\end{proof}

\begin{remark}
If we re-formulate this result in terms of $v=(N-2)w$ then $\tilde {\LL}u=e^{-2v/(N-2)}\big({\LL}u+\Gamma(v,u)\big)$ which converges to ${\LL}u+\Gamma(v,u)$ as $N\to\infty$. That is, conformal transformations in the limit $N\to\infty$ lead to drift transformations.  Note that as $N\to\infty$ the RHS of \eqref{conf-est-exp} tends to $\RR_\infty(u)-{\HH}_v(u,u)$  which is consistent with the well-known result for drift transformations.
\end{remark}

\bigskip

Conformal transformations are of particular interest in a Riemannian setting: if ${\LL}$ is the Laplace-Beltrami operator for the Riemannian manifold $(M,g)$ and if $N$ coincides with the dimension of $M$
then ${\LL}'$ is the Laplace-Beltrami operator for the Riemannian manifold $(M,f^{-2}\,g)$.

More precisely: Given an $n$-dimensional Riemannian manifold $(M,g)$ and a smooth function $w:M\to \R$, we can define a new Riemannian metric $\tilde g$ on $M$ by
$$\tilde g=e^{2w}\cdot g.$$
The induced (`new') Riemannian  volume measure is given by
$d\tilde m=e^{nw} \, dm$, the associated (`new') Dirichlet form is
$$\tilde\E(u)=\int_M |\nabla u|^2 e^{(-2+n)w}dm \quad \mbox{on }L^2(M, e^{nw}m).$$
Here $\nabla$, $|.|$ and $\Delta$ are all defined w.r.t. the metric $g$. Note that
$\tilde\nabla u=e^{-2w}\nabla u$ and thus $|\tilde\nabla u|_{\tilde g}^2=e^{-2w}\, |\nabla u|_g^2$.
Therefore, the associated (`new') Laplace-Beltrami operator is
$$\tilde\Delta u=e^{-2w}\Delta u-\frac{n-2}2\nabla e^{-2w}\cdot\nabla u.$$
The Ricci tensor for the metric $\tilde g$ is given by
\begin{eqnarray*}
\tilde\Ric(X,X)
&=& \Ric(X,X)-\Big(\Delta w+(n-2)|\nabla w|^2\Big)\cdot |X|^2-(n-2)\Big[ {\HH}_w(X,X)-(X\, w)^2\Big].
\end{eqnarray*}
(see e.g. \cite{Bes}, p 59). Applying this to the gradient of a (smooth) function $u$ on $M$ and taking into account that $\tilde\nabla u=e^{-2w}\nabla u$ yields
\begin{eqnarray}
\tilde\Ric(\tilde\nabla u,\tilde\nabla u)
&=& e^{-4w}\cdot \Big[\Ric(\nabla u,\nabla u)-\big(\Delta w+(n-2)|\nabla w|^2\big)\cdot |\nabla u|^2\nonumber\\
&&\qquad\qquad\qquad-(n-2)\big( {\HH}_w(\nabla u,\nabla u)-\langle \nabla w,\nabla u\rangle^2\big)\Big].
\end{eqnarray}
Thus, indeed,  \eqref{conf-est-exp} provides the \emph{exact} formula for the Ricci tensor for $\tilde {\LL}$.

\begin{example} The \emph{Poincar\'e model of hyperbolic space} is of the above type $(M,\tilde g)$ with
 $M=B_R(0)\subset \R^n$ and $\tilde g=f^{-2}\, g_{Euclid}$ for $$f(x)=\frac12\Big(1-\big({|x|}/R\big)^2\Big).$$
Its sectional curvature is constant $-\frac1{R^2}$ and the Ricci curvature is $-\frac{n-1}{R^2}\cdot \tilde g$.
\end{example}

\begin{remark}
For conformal transformations of  Laplacians  with drift on  smooth $N$-dimensional Riemannian manifolds,  estimates of the type
\begin{eqnarray*}
\tilde\Gamma_2(u)-\frac1N (\tilde {\LL} u)^2
\ge
e^{-4w}\,\Big[\K\, \Gamma(u)
-{\LL} w \,\Gamma(u) +c_1\Gamma(w)\, \Gamma(u) -(N-2){\HH}_{w}(u,u)+c_2\Gamma(w,u)^2\Big].
\end{eqnarray*}
have been presented in \cite{ThaWan} and \cite{Wan}, however, with wrong constants.
The first claim was $c_1= -N, c_2=2(N-2)$.
The `corrected' claim then was $c_1=-(N-4), c_2=N$.
Indeed, the correct choices are $c_1=-(N-2)$ and $c_2=N-2$.
\end{remark}

\section{Time Change and Drift Transformation}

This chapter will be devoted to
study the operator
\begin{equation}\label{time+drift}
{\LL}' u=f^2\, {\LL}u+f^2 \, \Gamma(h,u)
\end{equation}
for $f, h\in\A$. That is, in \eqref{Lsharp2} we specify to $r=1$ and $g=f$.
The case $h=-(N-2)\log f$ (`conformal transformation') was already treated in the previous chapter,
the cases $h=0$ (`time change') and $f=1$ (`drift transformation') will be considered in more detail in subsequent paragraphs of this chapter.
Any operator ${\LL}'$ of the form \eqref{time+drift} can be obtained   from ${\LL}$ by a combination
of
\begin{itemize}
\item a  drift transformation with $h$ and
\item  a time change with $f$.
\end{itemize}
Recall that $\Gamma'(u)=f^2\,\Gamma(u)$.  Theorem \ref{g2-est} yields a precise estimate for the $N'$-Ricci tensor associated to
the operator $\LL'$.
\begin{proposition} For every $N'>N$ and $u\in\A$
\begin{eqnarray*}
\RR'_{N'}(u)
&\ge&
f^4\, \RR_N(u)-\frac1{N'-N}\Big(\frac{N-2}2\Gamma(f^2,u)+f^2\Gamma(h,u)\Big)^2\\
&&+
\frac12\big(f^2{\LL}f^2-\Gamma(f^2)\big)\Gamma(u)-\frac{N-2}4\Gamma(f^2,u)^2- f^4 {\HH}_{h}(u,u)\\
&&+\frac12 f^2\Gamma(h,f^2)\Gamma(u)-f^2\Gamma(f^2,u)\Gamma(h,u).
\end{eqnarray*}
\end{proposition}
\begin{remark}
Let us re-formulate this estimate in the case $N'=\infty$ and  $f=e^{-w}$ for some $w\in\A$. It states that the Ricci tensor for the operator ${\LL}'=e^{-2w}\big[{\LL}+\Gamma(h,.)\big]$ satisfies
\begin{eqnarray}\label{dt-k-inf}
\RR'_\infty(u)&\ge& e^{-4w}\Big[\RR_N-
{\LL}w-\Gamma(w,h)\nonumber\\
&&-
\sup_{u\in\A}
\frac1{\Gamma(u)}\Big( (N-2)\Gamma(w,u)^2+ {\HH}_{h}(u,u)-2 \Gamma(w,u)\Gamma(h,u)\Big)\Big].
\end{eqnarray}
\end{remark}

\begin{corollary}\label{cor-g2-est}
Assume that the operator ${\LL}$ satisfies the $\BE(\K,N)$-condition. Then for every $N'>N$
 the operator ${\LL}'$ satisfies the $\BE(\K'_{N'},N')$-condition with
\begin{eqnarray}\label{dt-k-N}
\K'_{N'}&:=& f^2\,\K+
\frac12{\LL}f^2-2\Gamma(f)+f\Gamma(h,f)\nonumber\\
&&-
\sup_{u\in\A}
\frac1{\Gamma(u)}\Big[\frac1{N'-N}\Big((N-2)\Gamma(f,u)+f\Gamma(h,u)\Big)^2\\
&&\qquad\qquad +(N-2)\Gamma(f,u)^2+ f^2 {\HH}_{h}(u,u)+ \Gamma(f^2,u)\Gamma(h,u)\Big].\nonumber
\end{eqnarray}
\end{corollary}

\subsection{Drift Transformation}

Let us  have a closer look on the  case $f=1$. This is the  {`drift transformation'} which leads to a particularly simple, well-known formula for the Ricci tensor associated to the operator $\LL'=\LL+\Gamma(h,.)$. Obviously, $\Gamma'=\Gamma$.

\begin{proposition}
\begin{eqnarray}
\RR'(u)
&=&
\RR(u)-  {\HH}_{h}(u,u)
\end{eqnarray}
and for every $N'>N$
\begin{eqnarray*}
\RR'_{N'}(u)
&\ge&
\RR_N(u)-  {\HH}_{h}(u,u)-\frac1{N'-N}\Gamma(h,u)^2.
\end{eqnarray*}
\end{proposition}

\begin{proof} The ''$\ge$''-inequality follows immediately from Theorem \ref{g2-est}. The ''$\le$''-inequality in the case $N=\infty$ follows from another application of this result to the transformation of the operator $\LL'$ by means of the drift $\Gamma(-h,.)$ or, in other words, by exchanging the roles of $\LL$ and $\LL'$.
\end{proof}

\begin{corollary}[\cite{Bak-Eme}]
Assume that the operator ${\LL}$ satisfies the $\BE(\K,N)$-condition. Then for every $N'>N$
 the operator ${\LL}'={\LL}+\Gamma(h,.)$ satisfies the $\BE(\K',N')$-condition with
\begin{eqnarray*}
\K'&:=& \K-
\sup_{u\in\A}
\frac1{\Gamma(u)}\Big[ {\HH}_{h}(u,u)+\frac1{N'-N}\Gamma(h,u)^2\Big].
\end{eqnarray*}
In particular, if the operator
${\LL}$ satisfies the $\BE(\K,\infty)$-condition then
 the operator ${\LL}'={\LL}+\Gamma(h,.)$ satisfies the $\BE(\K',\infty)$-condition with
$\K'= \K-
\sup_{u}
\frac1{\Gamma(u)} {\HH}_{h}(u,u)$,
\end{corollary}

Actually, the framework of Theorem \ref{g2-est} allows to treat more general drift terms. Given $g_i,h_i\in\A$ for $i=1,\ldots,r$,
define  $Z:\A\to\A$ (regarded as `vector field') by
\begin{equation*}
Zu=\sum_{i=1}^r g_i\Gamma(h_i,u).
\end{equation*}
For instance, in the Riemannian case one might choose $r$ to be the dimension, $h_i=x^i$ the coordinate functions and $g_i=Z^i$ as the components of a given vector field $Z=\sum_i Z^i \frac{\partial}{\partial x^i}$.
According to Theorem \ref{g2-est},
\begin{eqnarray}
\RR'(u)
&=&
\RR(u)-  \big(DZ\big)(u,u)
\end{eqnarray}
where $\big(DZ\big)(u,u):=\sum_{i=1}^r \Big[g_i\, {\HH}_{h_i}(u,u)+\Gamma(g_i,u)\,\Gamma(h_i,u)\Big]$
and for every $N'>N$
\begin{eqnarray*}
\RR'_{N'}(u)
&\ge&
\RR_N(u)- \big(DZ\big)(u,u)-\frac1{N'-N}(Z\,u)^2.
\end{eqnarray*}
\begin{corollary}[\cite{Bak94}]\label{drift-be}
Assume that the operator ${\LL}$ satisfies the $\BE(\K,N)$-condition. Then for every $N'>N$
 the operator ${\LL}'={\LL}+Z$ satisfies the $\BE(\K',N')$-condition with
\begin{eqnarray*}
\K'&:=& \K-
\sup_{u\in\A}
\frac1{\Gamma(u)}\Big[ \big(DZ\big)(u,u)+\frac1{N'-N}(Zu)^2\Big].
\end{eqnarray*}
\end{corollary}

\begin{remark}
If ${\LL}$ is the generator of the diffusion process $(X_t,\PP_x)$ then under appropriate regularity assumptions (see e.g. \cite{RevYor}, \cite{FOT}) the transformed operator ${\LL}'={\LL}+Z$ will be the generator of the diffusion process $(X_t,\PP'_x)$ where $\PP'_x=M_t\cdot \PP_x$ on  $\sigma\{X_s: s\le t\}$ for a suitable martingale $(M_t)_{t\ge0}$, e.g. in the Riemannian case
$$M_t=\exp\left(\int_0^t Z(X_s)dX_x-\frac12 \int_0^t |Z(X_s)|^2ds\right).$$

A particular case is the  Doob transformation where $Z=2\nabla\log\phi$ for some $\phi>0$ satisfying ${\LL}\phi=0$. In this case, the martingale can alternatively be given as
$M_t=\phi(X_t)/\phi(X_0)$. The transition semigroup is given by $P'_tu=\frac1\phi\, P_t(\phi\, u)$.
\end{remark}

\subsection{Time Change}

Next we will focus on the particular case $h=0$. That is, we will consider the operator
$${\LL}'=f^2\, {\LL}$$
('time change') for some $f\in\A$. Obviously, $\Gamma'(u)=f^2\,\Gamma(u)$.
Theorem \ref{g2-est} immediately yields the following sharp estimate for the Ricci tensor for $\LL'$.

\begin{corollary}\label{g2-est-time}
\begin{eqnarray*}
\RR'_{N'}(u)
&\ge&
f^4\, \RR_N(u)-\frac{(N-2)(N'-2)}{N'-N}f^2\,\Gamma(f,u)^2+
\frac12\big(f^2{\LL}f^2-\Gamma(f^2)\big)\Gamma(u).
\end{eqnarray*}
\end{corollary}
\begin{remark}
If $f=e^{-w}$ for some $w\in\A$ these results can be reformulated   as 
\begin{eqnarray}
\RR'_{N'}(u)&\ge& e^{-4w}\Big[\RR_N(u) - {\LL}w\,\Gamma(u)- \frac{(N-2)(N'-2)}{N'-N}\,\Gamma(w,u)^2\Big].
\end{eqnarray}
\end{remark}

\begin{corollary}\label{be-tc}
Assume that the operator ${\LL}$ satisfies the $\BE(\K,N)$-condition. Then for every $N'>N$
the operator ${\LL}'$ satisfies the $\BE(\K',N')$-condition with
\begin{eqnarray}\label{k-time}
\K'&:=& f^2\,\K
+
\frac12{\LL}f^2-N^*\,\Gamma(f).
\end{eqnarray}
where $N^*=2+
\frac{[(N-2)(N'-2)]_+}{N'-N}\ge\max\{N,2\}$.

In particular,
the operator ${\LL}'$ satisfies the $\BE(\K'_\infty,\infty)$-condition with
\begin{eqnarray*}
\K'_\infty= f^2\,\K
+
\frac12{\LL}f^2-\max\{N,2\}\,\Gamma(f).
\end{eqnarray*}
\end{corollary}

\begin{remark} Assume that ${\LL}$ is the generator of the diffusion process $(X_t,\PP_x)$
in the sense that ${\LL}u(x)=\lim_{t\to0}\frac1t \EE_x\big( u(X_t)-u(X_0)\big)$ for all $u\in\A$ and all $x\in X$ or -- if there exists an invariant measure $m$ -- in the sense of $L^p$-convergence. Then under appropriate regularity assumptions (see e.g. \cite{RevYor}, \cite{FOT}) the transformed operator ${\LL}'=f^2\,{\LL} $ will be the generator of the diffusion process $(X'_t,\PP_x)$
where  $X'_t=X_{\tau(t)}$ with the 'time change' $t\mapsto\tau(t)$ being the inverse to $t\mapsto \int_0^t f^{-2}(X_s)ds$.
\end{remark}

\section{Dirichlet Forms}\label{df}

Let us from now on assume that $X$ is a measurable space, that all $u\in\A$ are bounded and measurable, and that we are given a $\sigma$-finite measure $m$ on $X$ with full support such that $\A\subset L^2(X,m)$.
(For the latter property, it might be of advantage not to require that the constants belong to $\A$.)
We say that
\begin{itemize}
\item $m$ is  ${\LL}$-\emph{invariant} if $\int {\LL}u\,dm=0$ for all $u\in\A$,
\item $m$ is  ${\LL}$-\emph{reversible} if $\int v{\LL}u\,dm=\int u{\LL}v\,dm$ for all $u,v\in\A$.
\end{itemize}
Throughout this chapter, let $h,w\in\A$ be given and  put ${\LL}'=e^{-2w}({\LL}+\Gamma(h,.))$ and $m'=e^{h+2w}m$. Then $m'$ is ${\LL}'$-invariant (or ${\LL}'$-reversible) if and only if
$m$ is ${\LL}$-invariant (or ${\LL}$-reversible, resp.).

Given a  measure $m$ on $X$ which is invariant and reversible w.r.t. ${\LL}$, we define the Dirichlet form $(\E,\Dom(\E))$ on $L^2(X,m)$
as the closure of
\begin{equation*}
\E(u)=\int\Gamma(u)\,dm\quad\mbox{for }u\in\A\subset L^2(X,m).
\end{equation*}
Similarly, we define
the Dirichlet form $(\E',\Dom(\E'))$ on $L^2\big(X,m\big)$
as the closure of
\begin{equation*}
\E'(u)=\int\Gamma(u)e^h\,dm\quad\mbox{for }u\in\A\subset L^2\big(X,e^{h+2w}m\big).
\end{equation*}
Then the generator $({\LL},\Dom(\LL))$ of $(\E,\Dom(\E))$ is the Friedrichs extension of $({\LL},\A)$ in $L^2(X,m)$ and
$({\LL'}, \Dom(\LL'))$, the generator  of $(\E',\Dom(\E'))$, is the Friedrichs extension of $({\LL}',\A)$ in $L^2(X,m')$.
\begin{definition}
We say that the triple $({\LL},\A,m)$ satisfies the $\BE(\K,N)$-condition if $\A$ is dense in $\Dom((-L)^{3/2})$  and if the operator $({\LL},\A)$ satisfies the $\BE(\K,N)$-condition.
\end{definition}
Density here is understood w.r.t.
the graph norm
$u\mapsto \Big[\| (-L)^{3/2}u  \|_{L^2}^2+\|u\|_{L^2}^2\Big]^{1/2}\,=\, \Big[\E(\LL u)+\|u\|_{L^2}^2\Big]^{1/2}$.

\begin{lemma}\label{D3dense}
If $\A$ is dense in $\Dom((-\LL)^{3/2})$ then it is also  dense in $\Dom((-\LL')^{3/2})$.
\end{lemma}

\begin{proof}
i) Firstly, the boundedness of $h$ and $w$ implies that $\Dom(\E)=\Dom(\E')$. In other words, $\Dom((-\LL)^{1/2})=\Dom((-\LL')^{1/2})$.

ii) Next, recall that $u\in \Dom(\LL')$ if (and only if) $u\in  \Dom((-\LL')^{1/2})$ and $\exists C$ s.t.
$\E'(u,\phi)\le C\cdot \|\phi\|_{L^2(m')}$ for all $\phi\in\Dom((-\LL')^{1/2})$.
For $u,\phi\in\A$ we easily see
$$\E'(u,\phi)=-\int e^h (\LL u+ \Gamma(h,u))\,dm\le C\cdot \|\phi\|_{L^2(m')}$$
for $C=\|e^h\|_{\infty} \cdot\big(\|\LL u\|_{L^2(m)}+\| \Gamma(h)\|_{\infty}\cdot \E(u)^{1/2}\big)$.
The  estimate $\E'(u,\phi)\le C\cdot \|\phi\|_{L^2(m')}$ is preserved if we approximate $u\in\Dom(\LL)$ and $\phi\in\Dom(\E)$ by $u_n\in\A$ and $\phi_n\in\A$, resp.
This proves $\Dom(\LL)\subset\Dom(\LL')$. Exchanging the roles of $\LL$ and $\LL'$ yields the converse inclusion.
That is, $\Dom(\LL)=\Dom(\LL')$.

iii) Observe that $u\in\Dom((-\LL)^{1/2})\Leftrightarrow e^{h/2}u\in\Dom((-\LL)^{1/2})$ and that
$$u\in\Dom(-\LL)\Leftrightarrow e^{h/2}u\in\Dom(-\LL).$$
To see the latter, note that
$\E'(u,\phi)=\E(e^{h/2}u,e^{h/2}\phi)+\int u\phi e^{h/2}\LL e^{h/2}\,dm$.

iv) Our next claim is that
$$u\in\Dom((-\LL')^{3/2})\Leftrightarrow e^{h/2}u\in\Dom((-\LL)^{3/2}).$$
To prove this, recall that
$u\in\Dom((-\LL')^{3/2})$ if (and only if) $u\in  \Dom(-\LL')$ and $\exists C$ s.t.
\begin{equation}\label{dom-l3}
\int \LL' u\cdot \LL' \phi\,dm'\le C\cdot \E'(\phi)^{1/2}
\end{equation} for all $\phi\in\Dom(-\LL')$.
For $u,\phi\in\A$ put $\tilde u=e^{h/2}u$, $\tilde\phi=e^{h/2}\phi$. Then
\begin{eqnarray*}
\int \LL' u\cdot \LL'\phi\,dm'&=&
\int\big(\LL u+\Gamma(h,u)\big)\cdot \big(\LL\phi+\Gamma(h,\phi)\big)\cdot e^{h-2w}\,dm\\
&=&
\int\big(\LL\tilde u - u\LL e^{h/2 }\big)\cdot \big(\LL\tilde\phi - \phi\LL e^{h/2 }\big)\cdot e^{-2w}\,dm\\
&=& \int \LL\tilde u \cdot\LL \tilde\phi\cdot e^{-2w}dm+ \mbox{LOT}_1\\
&=&-\int \Gamma (\LL\tilde u ,\tilde\phi)\, e^{-2w}dm+ \mbox{LOT}_2\\
&\le& C\cdot \|\tilde u\|_{\Dom((-\LL)^{3/2})}\cdot  \E(\tilde \phi)^{1/2}
\end{eqnarray*}
where $\mbox{LOT}_1$ and $\mbox{LOT}_2$ denote 'low order terms' which can be estimated in terms of $\|u\|_{L^2(m)}$, $\E(u)$, $\|\LL u\| _{L^2(m)}$
and $\E(\phi)$. This proves \eqref{dom-l3} for all $u$ and $\phi\in\A$.
Due to the assumed density of $\A$ in all the $\Dom((-\LL)^{k/2})$ for $k=1,2,3$ and the previously proven equivalences
$\Dom((-\LL)^{1/2})=\Dom((-\LL')^{1/2})$ and
 $\Dom(\LL)=\Dom(\LL')$, the estimate \eqref{dom-l3}
 extends to all $u\in   \Dom(-\LL')$ and all $\phi\in\Dom(-\LL')$.
 This proves the implication ''$\Rightarrow$'' of the above claim. Again the converse implication follows by interchanging the roles of $\LL$ and $\LL'$.

v) Finally, we will prove that $\A$ is dense in $\Dom((-\LL')^{3/2})$. Let $u\in \Dom((-\LL')^{3/2})$ be given, put
$\tilde u=e^{h/2}u\in \Dom((-\LL')^{3/2})$ and choose an approximating sequence $\tilde u_n\in\A$ such that
$\E(\LL(\tilde u -\tilde u_n))+\int (\tilde u -\tilde u_n)^2\,dm \to0$ for $n\to\infty$.
But this already implies
$\E'(\LL'( u - u_n))+\int (u -u_n)^2\,dm \to0$ for $u_n=e^{-h/2}\tilde u_n\in\A$ since for every $\phi\in \A$
\begin{eqnarray*}
\E'(\LL'\phi)&=&
\int\Gamma\Big(e^{-2w}\big(\LL \phi+\Gamma(h,\phi)\big)\Big)e^h\,dm\\
&=&\int \Gamma\Big(e^{-2w}\big(e^{-h/2}\LL \tilde\phi-\tilde\phi \LL e^{h/2}\big)\Big)e^h\,dm\\
&\le& C_1\cdot \E(\LL \tilde \phi)+C_2 \cdot \int \tilde\phi^2\,dm
\end{eqnarray*}
where $\tilde\phi=e^{h/2}\phi$.
\end{proof}

Many spectral properties for $(\LL,\Dom(\LL))$ and functional inequalities involving it will follow -- typically with sharp constants -- from the Bakry-{\'E}mery estimate $\BE(\K,N)$ for $(\LL,\A,m)$, among them Poincar\'e inequalities, Sobolev and logarithmic Sobolev inequalities, concentration of measure estimates, isoperimetric inequalities, gradient estimates and heat kernel estimates.
We refer to the surveys \cite{Bak94} and \cite{Led2}. To mention at least one example, we state a fundamental  estimate for the spectral gap.

\begin{proposition}
Assume that the  $(\LL,\A,m)$ satisfies the $\BE(\K,N)$-condition. Then the spectral gap $\lambda$ of the operator $({\LL'},\Dom(\LL'))$ on $L^2(X,m')$ satisfies
\begin{equation*}
\lambda\ge\inf_{x\in X}\K'_\infty(x)
\end{equation*}
with the function $\K'_\infty$ from \eqref{dt-k-N}
where $f=e^{-w}$.
A more refined estimate yields
$\lambda\ge\frac{N'}{N'-1}\cdot \inf_{x\in X}\K'_{N'}(x)$
for every $N'>N$
with the function $\K'_{N'}$ from \eqref{dt-k-N}.
\end{proposition}

\begin{proof}
According to Corollary \ref{cor-g2-est},
$({\LL}',\A)$ satisfies the $\BE(\mathrm{K}',N')$-condition with the function $\K'$ given by \eqref{dt-k-N}.
Finally, according to Lemma \ref{D3dense}, $\A$ is dense in $\Dom((-\LL')^{3/2})$.
Thus $(\LL',\Dom(\LL'))$ satisfies the
$\BE(\mathrm{K}',N')$-condition with $\mathrm{K}':=\inf_{x\in X}\K'(x)$.
According to \cite{Bak-Eme}, every $\BE(\mathrm{K}',N')$-condition with a constant $\mathrm{K}'>0$ implies a spectral gap estimate of Lichnerowicz type $\lambda\ge\frac{N'}{N'-1}\cdot \mathrm{K}'$.
\end{proof}

The assertion of Corollary \ref{g2-est-time} also allows for a subtle generalization of the well-known Bonnet-Myers Theorem.
\begin{proposition} Assume that the  $(\LL,\A,m)$ satisfies the $\BE(\K,N)$-condition and that there exist numbers $N^*>N, N^*\ge2$, $\mathrm{K}>0$ and a function $f\in\A$, $|f|\le 1$ such that
\begin{eqnarray}\label{mod-myers}
f^2\,\K
+
\frac12{\LL}f^2-N^*\,\Gamma(f)\ge \mathrm{K}.
\end{eqnarray}
 Then
\begin{eqnarray}
\diam(X)\le \frac\pi{\sqrt{\mathrm{K}}}\cdot\sqrt{N-1+ \frac{(N-2)^2}{N^*-N}}.
\end{eqnarray}
where $\diam(X)=\sup\{u(x)-u(y):\ u\in\A, \Gamma(u)\le 1\}$ denotes the diameter of $X$ w.r.t. the `intrinsic metric' induced by $({\LL},\A)$.
\end{proposition}
The choice $f=1$ will lead (in the limit $N^*\to\infty$) to the classical Bonnet-Myers Theorem.

\begin{proof} Put $N'=N+\frac{(N-2)^2}{N^*-N}$ in the case $N\not=2$ and $N'=3$ in the case $N=2$.
According to Corollary \ref{be-tc}, the operator ${\LL}'=f^2 {\LL}$ satisfies $\BE(\K',N')$ with $\K'$ given by \eqref{k-time}.
Together with the assumption \eqref{mod-myers} this yields the $\BE(\mathrm{K},N')$-condition for ${\LL}'$.
According to \cite{BakLed96} this implies the diameter bound w.r.t. the intrinsic metric induced by ${\LL}'$.
Due to the assumption $|f|\le1$ the latter is bounded by the intrinsic metric induced by ${\LL}$.
\end{proof}

\section{Smooth Metric Measure Spaces}

Finally, we will study curvature bounds for metric measure spaces and their behavior under transformation of the data.
A triple $(X,d,m)$ is called \emph{metric measure space} if $(X,d)$ is a complete separable metric space and if $m$ is a locally finite measure on the Borel field of $X$. Without restriction we will always assume that $m$ has full topological support.

\begin{definition}
Given (extended) numbers $\KK\in\R$ and $N\in[1,\infty]$ we say that $(X,d,m)$ satisfies the \emph{entropic curvature-dimension condition} $\CD^e(\KK,N)$ if the Boltzmann entropy
\begin{eqnarray*}
S: \mu\mapsto\left\{
\begin{array}{ll}
\int\left(\frac{d\mu}{dm}\right)\log\left(\frac{d\mu}{dm}\right)\,dm&\quad \mbox{if }\mu\ll m,\\
+\infty&\quad \mbox{else.}
\end{array}\right.
\end{eqnarray*}
is $(\KK,N)$-convex on the $L^2$-Wasserstein space $({\mathcal P}_2(X), d_W)$.
\end{definition}
Here a function $S$ on a metric space $(Y,d_Y)$ is called $(\KK,N)$-convex if every pair of points $y_0,y_1\in Y$ can be joined by a (minimizing, constant speed) geodesic $\big(y(t)\big)_{0\le t\le1}$ in $Y$ such that the function $u(t)=e^{-\frac1N S(y(t))}$ is lower semicontinuous in $t\in[0,1]$, continuous in $(0,1)$ and satisfies
$$u''\le -\frac KN |\dot y|^2\cdot u$$
weakly in $(0,1)$.
In the limit $N\to\infty$ this leads to the usual $\KK$-convexity; if in addition $\KK=0$ it yields the classical convexity.
In the general case, $(\KK,N)$-convexity gives a precise meaning for weak solutions to the differential inequality
$$D^2 S-\frac1N \, DS\,\otimes\, DS\ge K$$
on geodesic spaces.

Note that the entropic curvature-dimension condition
implies that $(X,d,m)$ is a geodesic space. More precisely,
$d(x,y)
=\inf\Big\{ \int_0^1 |\dot\gamma_t|\,dt: \ \gamma:[0,1]\to X \mbox{ rectifiable, } \gamma_0=x, \gamma_1=y\Big\}$ for each $x,y\in X$.

We want to prove that the entropic curvature-dimension condition is preserved (with modified parameters) under the most natural transformations of the data $d$ and $m$. And we want to analyze how the parameters $\KK$ and $N$ will change.
The transformation which we have in mind are
\begin{itemize}
\item
given a measurable function $v$ on $X$, we replace the measure $m$ by the weighted measure with Radon-Nikodym derivative $e^v$:
\begin{equation*}
m'=e^v\,m;
\end{equation*}
\item
given a function $w$ on $X$, we replace the length metric $d$ by the weighted length metric with conformal factor $e^w$:
\begin{equation}\label{d-weighted}
d'(x,y)
=\inf\Big\{ \int_0^1 |\dot\gamma_t|\cdot e^{w(\gamma_t)}\,dt: \ \gamma:[0,1]\to X \mbox{ rectifiable, } \gamma_0=x, \gamma_1=y\Big\}.
\end{equation}
\end{itemize}

Treating these questions in full generality is beyond the scope of this paper. We will restrict ourselves here to \emph{smooth} metric measure spaces which allows to benefit from the results of the previous chapters. The general case
requires to deal with subtle regularity issues. We refer to \cite{EKS} and \cite{Sav} for such  approximation and smoothing procedures in the general case.

\begin{definition} A metric measure space $(X,d,m)$ is called \emph{smooth} if there exists a diffusion operator $\LL$ defined on an algebra $\A$ as above (chapter 2) such that $\A\subset L^2(X,m)$ and
\begin{itemize}
\item $m$ is a reversible invariant measure for $({\LL},\A)$
and $\A$ is dense in $\Dom((-\LL)^{3/2})$;
\item $d$ is the intrinsic metric for ${\LL}$, i.e. for all $x,y\in X$
$$d(x,y)=\sup\{ u(x)-u(y): \  u\in\Dom(\E)\cap {\mathcal C}_b(X), \hat\Gamma(u)\le m\}.$$
\end{itemize}
Here $\hat\Gamma(.)$ denotes the so-called energy measure, i.e. the measure-valued quadratic form on
$\Dom(\E)\cap L^\infty(X)$ extending the quadratic form
 $\Gamma(.)$ on $\A$.
\end{definition}
In particular, each $u\in\A$ will be  bounded and continuous.

\begin{theorem} Let $(X,d,m)$ be a smooth metric measure space. 
Given $v,w\in\A$ define  $m'=e^v\,m$  and $d'$ as in
\eqref{d-weighted}.
 If $(X,d,m)$ satisfies the entropic curvature-dimension condition $\CD^e(\KK,N)$ for constants $N\in[1,\infty)$ and $\KK\in\R$ then for each  $N'\in(N,\infty]$ the metric measure space $(X,d',m')$  satisfies the entropic curvature-dimension condition $\CD^e(\KK',N')$ for
\begin{eqnarray}\label{mms-N}
\KK'&=& \inf_X \, e^{-2w}\Big[\KK-
{\LL}w+\Gamma(w,2w-v)\nonumber\\
&&\qquad\qquad-
\sup_{u\in\A}
\frac1{\Gamma(u)}\Big( \frac1{N'-N}\Gamma(v-Nw,u)^2+(N-2)\Gamma(w,u)^2\\
&&\qquad\qquad\qquad\qquad\qquad\qquad\qquad+ {\HH}_{v-2w}(u,u)-2 \Gamma(w,u)\Gamma(v-2w,u)\Big)\Big].\nonumber
\end{eqnarray}
If $w=0$ also $N=N'=\infty$ is admissible; if $w=\frac1N v$ also $N'=N$ is admissible.
\end{theorem}

\begin{remark} The case $w=0$ of a pure measure transformation (or 'drift transformation') is well studied, see \cite{Stu1,Stu2, LV}.
Thus let us briefly focus on the case $v=0$ of a pure metric transformation.
In this case, formula \ref{mms-N} simplifies to
\begin{eqnarray*}
\KK'&=& \inf_X \, e^{-2w}\Big[\KK-
{\LL}w+2\Gamma(w)\nonumber\\
&&\qquad\qquad-
\sup_{u\in\A}
\frac1{\Gamma(u)}\Big( \big(\frac{N'\,N}{N'-N}+2\big)\Gamma(w,u)^2-2 {\HH}_{w}(u,u)\Big)\Big]
\end{eqnarray*}
or in terms of $f=e^{-w}$
\begin{eqnarray}\label{mms-N-no-measure}
\KK'&=& \inf_X \, \Big[\KK f^2+\frac12
{\LL}f^2-
\sup_{u\in\A}
\frac1{\Gamma(u)}\Big( \big(\frac{N'\,N}{N'-N}-
2\big)\Gamma(f,u)^2+ {\HH}_{f^2}(u,u)\Big)\Big].
\end{eqnarray}
In the case $N'=\infty$, the expression $\frac{N'\,N}{N'-N}$ simplifies to $N$.
\end{remark}

\begin{proof}
If $N=\infty$, the only admissible choice is $N'=\infty$ and $w=0$. This drift  transformation is covered by \cite{Stu1, LV}.
Thus throughout the rest $N<\infty$.

Firstly, we then observe that the $\CD^e(\KK,N)$-condition implies that the underlying space is locally compact (\cite{EKS}, Prop. 3.6).
This guarantees that $(X,d,m)$ satisfies the criteria of \cite{AGS}, Def. 3.6, Def. 3.13.
Thus secondly, we conclude that the Dirichlet form $(\E,\Dom(\E))$ on $L^2(X,m)$ induced by the operator $(\LL,\A)$ (as considered in chapter \ref{df})
coincides with the Cheeger energy on $L^2(X,m)$ induced by the metric $d$ (\cite{AGS}, Thm 3.14).
In particular, $(X,d,m)$ is infinitesimally Hilbertian.
According to \cite{EKS}, Cor. 5.1, the condition $\CD^e(\KK,N)$ implies the Bakry-Ledoux gradient estimate from Lemma \ref{bochner-lemma} below
and thus $(\LL,\A)$ satisfies the Bakry-{\'E}mery condition $\BE(\KK,N)$.
According to Corollary \ref{cor-g2-est} (with $h=v-2w$), therefore, $(\LL',\A)$ satisfies
the Bakry-{\'E}mery condition $\BE(\KK',N')$ for any $N'>N$ and $\KK'$ given by
\eqref{mms-N}.

Obviously, the Dirichlet forms $\E'$ and $\E$ (as well as the measures $m$ and $m'$) are comparable and
$\hat\Gamma'(.)=e^{v-2w}\cdot\hat\Gamma(.)$. 
Thus
\begin{eqnarray*}
d_{\E'}(x,y)&=&\sup\Big\{
 u(x)-u(y): \  u\in\Dom(\E)\cap {\mathcal C}_b(X), \hat\Gamma(u)\le e^{2w}\, m\Big\}.
\end{eqnarray*}
Moreover,  the intrinsic metrics for  both Dirichlet forms are length metrics with the same set of rectifiable curves. For each rectifiable curve $\gamma:[0,1]\to X$ its length w.r.t. the metric $d_{\E'}$ therefore is 
\begin{eqnarray*}
\mathrm{Length}_{\E'}(\gamma)&=& \int_0^1 |\dot\gamma_t|\cdot e^{w(\gamma_t)}\,dt.
\end{eqnarray*}
Thus $d_{\E'}$ coincides with the metric $d'$ as defined in \eqref{d-weighted}.

By assumption $\A$ is dense in $\Dom(({-\LL)^{3/2}})$. Thus according to Lemma \ref{D3dense}
it is also dense in $\Dom(({-\LL')^{3/2}})$.
Thus again by Lemma \ref{bochner-lemma}
 the Bakry-{\'E}mery condition $\BE(\KK',N')$ is equivalent  to the entropic curvature-dimension condition $\CD^e(\KK',N')$ for the smooth mms $(X,d',m')$.\end{proof}
To summarize
$$\CD^e(\KK,N)\mbox{ for }(X,d,m)\ \Leftrightarrow\ \BE(\KK,N)\mbox{ for }\LL\ \Rightarrow \ \BE(\KK',N')\mbox{ for }\LL'\ \Leftrightarrow \ \CD^e(\KK',N')\mbox{ for }(X,d,m).$$

\begin{lemma}\label{bochner-lemma} For any smooth metric measure space $(X,d,m)$ and any $\KK\in\R$, $N\in[1,\infty)$ the following are equivalent
\begin{itemize}
\item[(i)] $({\LL},\A)$ satisfies the Bakry-{\'E}mery condition $\BE(\KK,N)$;
\item[(ii)] $\forall u\in\Dom(({-\LL)^{3/2}})$
and $\forall \phi\in\Dom({\LL})\cap L^\infty(X,m)$ with $\phi\ge0$, ${\LL}\phi\in L^\infty(X,m)$;
\begin{equation}\label{bochner1}
\frac12\int {\LL}\phi\cdot \Gamma(u)\,dm -\int \phi \Gamma(u,{\LL}u)\,dm\ge \KK\int\phi \Gamma(u)\,dm+\frac1N \int \phi ({\LL}u)^2\,dm;
\end{equation}
\item[(iii)] $\forall u\in\Dom(\E)$ with bounded $\Gamma(u)$ and $\forall t>0$
\begin{equation*}
\Gamma(P_tu)+  \frac{4\KK t^2}{N(e^{2\KK t}-1)} ({\LL}P_tu)^2\le e^{-2\KK t}P_t\Gamma(u);
\end{equation*}
\item[(iv)] $(X,d,m)$ satisfies the entropic curvature-dimension condition $\CD^e(\KK,N)$.
\end{itemize}
\end{lemma}

If we assumed that the algebra $\A$ were invariant under the semigroup $P_t$ then the implication of (i) $\Rightarrow$ (iii) would be more or less standard.
Following \cite{Led2} we could conclude (iii) for all $f\in\A$ by a simple differentiation/integration argument. Then (iii) in full generality would follow by a straightforward density argument.
However, assuming that $\A$ is invariant under $P_t$ in general is too restrictive.

Our main challenge will be to verify the Bochner inequality with parameters $\KK$ and $N$ for a `large' class of functions  which contains $\A$ and which is invariant under $P_t$. This is property (ii).

\begin{proof}
The equivalence of (ii), (iii) and (iv) was proven in \cite{EKS}.
%
The implication (ii)$\Rightarrow$(i) is trivial.
To proof the converse, let us assume (i).
Multiplying this pointwise inequality for $u\in\A$ by a nonnegative $\phi$ and integrating w.r.t. $m$ yields
\begin{equation*}\label{bochner2}
\frac12\int \phi\cdot {\LL}\Gamma(u)\,dm -\int \phi \Gamma(u,{\LL}u)\,dm\ge \KK\int\phi \Gamma(u)\,dm+\frac1N \int \phi ({\LL}u)^2\,dm.
\end{equation*}
For $\phi\in\Dom(\LL)$, the symmetry of $L$ then yields
\eqref{bochner1} for all $u\in\A$.

By assumption, the algebra $\A$ is dense in
$\Dom((-\LL)^{3/2})$.
Any $u\in \Dom((-\LL)^{3/2})$, therefore can be approximated by $u_n\in \A$ such that
$\Gamma(u_n)\to \Gamma(u)$, $(\LL u_n)^2\to (\LL u)^2$  and $\Gamma(u_n,\LL u_n)\to \Gamma(u,\LL u)$ in $L^1(X,m)$.
Hence,
 we may pass to the limit  in inequality \eqref{bochner1}. This proves the claim.
\end{proof}

\medskip

{\it Notes added in proof.} After finishing (the first version of) this paper, the new monograph \cite{BGL} appeared which contains at various places calculations (e.g. section 6.9.2 or C.6) similar to those in the current paper. However, none of our main results is obtained there.

We also would like to mention the follow-up work by Bang-Xian Han and Anna Mkrtchyan \cite{HM}
which extend several results of this paper to the setting of metric measure spaces without `smoothness' assumptions.


\begin{thebibliography}{999}


\bibitem{AGS}
\textsc{Ambrosio, L. and Gigli, N. and Savar\'{e}, G.},
Bakry-\'{E}mery curvature-dimension condition and {R}iemannian {R}icci curvature bounds,
Preprint  arXiv:1209.5786
(2012).



\bibitem{Bak94}
\textsc{Bakry, D.},
L'hypercontractivit\'e et son utilisation en th\'eorie des semigroupes, pp. 1-114 in
\textit{Lectures on probability theory, Saint-Flour 1992},
Lecture Notes in Math., 1581, Springer, 1994.


\bibitem{BakLed96}
\textsc{Bakry, D., Ledoux, M.},
Sobolev inequalities and Myers's diameter theorem for an abstract Markov generator,
Duke Math. J. {\bf 85},  253 -- 270 (1996).

\bibitem{Bak-Eme}
\textsc{Bakry, D. and {\'E}mery, M.},
 {Diffusions hypercontractives},
 {S\'eminaire de probabilit\'es, {XIX}, 1983/84},
  {Lecture Notes in Math.} {\bf 1123},     {177--206},
{Springer},  {Berlin}     {1985}.

\bibitem{BGL}
\textsc{Bakry, D., Gentil, I. and Ledoux, M.},
\textit {Analysis and Geometry of Markov Diffusion Operators},
Grundlehren 348,
{Springer},  {Berlin}     {2014}.

\bibitem{Bak-Qia}
\textsc{Bakry, D. and Qian, Z.},
Some new results on eigenvectors via dimension, diameter, and Ricci curvature,
Advances in Math. {\bf 155}, 98-153 (2000).

\bibitem{Bes}
\textsc{Besse, A.L.},
\textit{Einstein manifolds}. Springer 1987.

\bibitem{EKS}
\textsc{Erbar, M. and Kuwada, K. and Sturm, K.T.}
On the equivalence of the entropic curvature-dimension condition and Bochner's inequality on metric measure spaces,
Preprint arXiv:1303.4382(2013).

\bibitem{FOT}
\textsc{Fukushima, M. and  Oshima, Y. and Takeda, M.},
\textit{Dirichlet forms and symmetric Markov processes}. De Gruyter 2011.

\bibitem{HM}
\textsc{Han, B.-X.. and Mkrtchyan, A.},
\textit{Conformal transformation on metric measure spaces},
 {Preprint  arXiv:1511.03115}.


\bibitem{Led2}
\textsc{Ledoux, M.},
\textit{The geometry of Markov diffusion generators}.
ETHZ, Institute for Mathematical Research, 1998.

\bibitem{LV}
\textsc{Lott, J. and Villani, C.}
Ricci curvature for metric-measure spaces via optimal transport,  Annals of Math. {\bf 169},  903-991
(2009).

\bibitem{RevYor}
\textsc{Revuz, Yor, M.},
\textit{Continuous martingales and Brownian motion}. Springer 1999.

\bibitem{Sav}
\textsc{Savar{\'e}, G.},
Self-improvement of the Bakry-{\'E}mery condition and Wasserstein contraction of the heat flow in RCD(K,$\infty$) metric measure spaces,
 to appear in Disc. Cont. Dyn. Sist. A.
(2013) 1-23


\bibitem{Stu1}
\textsc{Sturm, K.T.},
    {On the geometry of metric measure spaces. {I}},
  {Acta Math.} {\bf 196}
    {(2006)}, {65--131}.

		
\bibitem{Stu2}
\textsc{Sturm, K.T.},
    {On the geometry of metric measure spaces. {II}},
  {Acta Math.} {\bf 196}
    {(2006)}, {133--177}.
		
\bibitem{ThaWan}
\textsc{Thalmaier, A. and Wang, F.-Y.}
Gradient estimates for harmonic functions on regular domains in Riemannian manifolds,
Journal of Functional Analysis {\bf 155} (1998) 109--124.



\bibitem{Wan}
\textsc{Wang, F.-Y.}
Estimates of the first Neumann eigenvalue and the log-Sobolev constant on non-convex manifolds, Math. Nach.
{\bf 280} (2007),1431--1439.
\end{thebibliography}
\end{document}